\documentclass{elsarticle}
\usepackage[utf8]{inputenc}
\usepackage{amsmath,amssymb,amsthm,amsfonts}
\usepackage{bm}
\usepackage[hidelinks,colorlinks = true]{hyperref}
\usepackage{enumitem}
\usepackage{blkarray}
\usepackage{mathtools}
\setenumerate{labelindent=0.25\parindent,leftmargin=*}
\DeclareMathOperator{\rk}{rk}

\DeclareMathOperator{\rad}{rad}

\DeclareMathOperator{\Char}{char}

\let\det\undefined
\DeclareMathOperator{\det}{det}

\DeclareMathOperator{\sgn}{sgn}

\def\myjoin{=\joinrel =\joinrel =}
\def\mylongjoin{=\joinrel =\joinrel =\joinrel =\joinrel =\joinrel =}
\usepackage{tcolorbox}
\newtcolorbox{mybox}{colback=red!5!white,colframe=red!75!black, sharp corners = all}
\usepackage{bclogo} 
\def\ucsign{\bcpanchant}

\theoremstyle{plain}
\newtheorem{theorem}{Theorem}[section]
\newtheorem*{theorem*}{Theorem}
\newtheorem{lemma}[theorem]{Lemma}

\newtheorem{corollary}[theorem]{Corollary}

\theoremstyle{definition}
\newtheorem{definition}[theorem]{Definition}

\newtheorem{remark}[theorem]{Remark}
\newtheorem{note}[theorem]{Note}
\newtheorem{example}[theorem]{Example}

\numberwithin{equation}{section}

\journal{Linear Algebra and its Applications}

\begin{document}

\title{Linear maps preserving the Cullis'
determinant. I}
\author{Alexander Guterman}
\ead{alexander.guterman@biu.ac.il}
\author{Andrey Yurkov\corref{cor1}}
\ead{andrey.yurkov@biu.ac.il}
\cortext[cor1]{Corresponding author}

\affiliation{organization={Department of Mathematics, Bar Ilan University},addressline={Ramat-Gan},postcode={5290002}, country={Israel}}


\begin{abstract}
This paper is the first in the series of papers devoted to the explicit
description of linear maps preserving the Cullis' determinant of rectangular
matrices with entries belonging to an arbitrary ground field which is large
enough.

The Cullis' determinant is defined for every matrix of size
$n\times k$, where $n \ge k \ge 1$ and is equal to the ordinary determinant
if $n = k$. In this paper we solve the linear preserver problem for the
Cullis' determinant for $k \ge 4, n \ge k + 2$ and $n + k$ is even. It
appears that in this case all linear maps preserving the Cullis' determinant
are non-singular and could be represented by two-sided matrix multiplication.

We also obtain the complete description of linear maps preserving the Cullis'
determinant of rectangular matrices for $k = 1$, $n\ge k$ and provide an
example showing that for $k = 2, n \ge k + 2 = 4$ a linear map preserving
$\det _{n\,2}$ may not correspond to the description of linear maps preserving
$\det _{n\,k}$ for $k \ge 3$ obtained in this paper and in its sequel.

Note that the cases where $n = k$ or $n = k + 1$ admit slightly different
description allowing (sub)matrix transposition and were completely studied
before: the case where $n = k$ is a classical linear preserver problem
for the ordinary determinant and was solved by Frobenius; the complete
characterization for the case where $n = k + 1$ was obtained in the previous
paper by the authors.
\end{abstract}
%
%

%
%

\maketitle

\section{Introduction}
\label{sec1}

The determinant of a matrix is a classical object of investigations in
Linear Algebra and its applications. Usually only the determinant of square
matrices is considered, but different attempts to generalize the notion
of determinant to the set of rectangular matrices have been made for a
long time. Cullis introduced the concept of determinant (he called it
\emph{determinoid}) of a rectangular matrix in his monograph~\cite{cullis1913}
and it is presumably the first published generalization of the determinant
to the rectangular case. Several properties known for the ordinary determinant
are studied and shown to be valid for the Cullis determinant in~\cite[\S 5, \S 27, \S 32]{cullis1913},
and we recall some of them below. Algebraic characterization of the Cullis
determinant can be found in \cite{amiri2010,makarewicz2014}.

In 1966 Radi\'{c}~\cite{radic1966} independently proposed a definition
of the determinant of a rectangular matrix, which is equivalent to the
Cullis' definition, and accordingly, in some papers it is called
\emph{Radi\'{c}'s determinant}~\cite{amiri2010} or
\emph{Cullis-Radi\'{c} determinant}~\cite{makarewicz2014,makarewicz2020}.
After that there were several other generalizations of the determinant
of a square matrix to rectangular matrices given, for example, in~\cite{pyle1962,sudhir2014,yanai2006}.

The notion of determinant has been studied in many contexts and one of
them is the investigation of linear maps preserving the determinant. The
first result in this direction dates back to 1897 and is due to Frobenius~\cite{GF}.

\begin{theorem}[Frobenius, {\cite[\S 7, Theorem I]{GF}}]%
\label{thm:1}
Let
$S\colon {\mathcal M}_{n}(\mathbb C) \to {\mathcal M}_{n}(\mathbb C)$ be
a bijective linear map satisfying $\det (S(X))=\det (X)$ for all
$X \in {\mathcal M}_{n} (\mathbb C)$, where $\mathbb C$ denotes the field
of complex numbers. Then there exist matrices
$M, N \in {\mathcal M}_{n}(\mathbb C)$ with $\det (MN) = 1$ such that
\begin{equation*}
S(X) = MXN\;\;\text{for all}\; X \in {\mathcal M}_{n} (\mathbb C) \;\;
\text{or}\;\; S(X) = MX^{t}N\;\;\text{for all}\; X \in {\mathcal M}_{n}(
\mathbb C).
\end{equation*}
\end{theorem}

This result by Frobenius prompted the investigation of so-called linear
preserver problems concerning the characterization of linear operators
on matrix spaces that leave certain functions, subsets, relations, etc.,
invariant. The research started by Frobenius was continued in the works
by Schur, Dieudonn\'e, Dynkin and others. One may see~\cite{hogben_handbook_2014}
for a brief introduction to the subject and~\cite{LAMA199233} for an extensive
survey.

A question about generalization of  the Frobenius' theorem to the Cullis'
determinant on general rectangular matrices arises naturally. The corresponding
linear preserver problem was not even partially solved until 2024. In a
recent paper~\cite{Guterman2024}, the authors initiated research on this
problem by considering the case of rectangular matrices of size
$n\times k$ with $n = k + 1$.

Continuing the work in this direction, we found the solution of the linear
preserver problem of the Cullis' determinant for the matrices of all possible
sizes $n \times k$ with $k \neq 1,2$ in the case where the ground field
is large enough. Note that the cases where $k = 1$ and $k = 2$ are special
because if $k = 1$, then the Cullis' determinant is just a linear functional
and if $k = 2$, then it is just a quadratic form where the standard methods
of studying quadratic forms may be applied. We discuss these cases in detail
in Section~\ref{sec:K1} and Section~\ref{sec:K2}, correspondingly.

Our result contains several cases which differ from each other by the approach
and techniques applied. In this part we consider the case where
$k \ge 4, n \ge k + 2$ and $n + k$ is even together with the assumption that
the size of the ground field is greater than $k$. The results are likely
valid for every ground field but different methods are required. We also
provide an example showing that already in the case where $k = 3$ another
approach is needed (this case will be considered in the separate paper).

Let us provide the basic notation used throughout this paper.

By ${\mathbb F}$ we denote a field without any preliminary restrictions
on its cardinality and characteristic. We will use an inequality
$|{\mathbb F}| > k$ for $k \in \mathbb N$ which means that either
${\mathbb F}$ is infinite or ${\mathbb F}$ is finite and
$|{\mathbb F}| > k$.

We denote by ${\mathcal M}_{n\, k}({\mathbb F})$ the set of all
$n\times k$ matrices with the entries from a certain field
${\mathbb F}$. $O_{n\, k}\in {\mathcal M}_{n\, k}({\mathbb F})$ denotes
the matrix with all the entries equal to zero.
$I_{n\, n} = I_{n} \in {\mathcal M}_{n\, n}({\mathbb F})$ denotes an identity
matrix. Let us denote by
$E_{ij} \in {\mathcal M}_{n\, k}(\mathbb F)$ a matrix, whose entries are
all equal to zero besides the entry on the intersection of the $i$-th row
and the $j$-th column, which is equal to one. By $x_{i\, j}$ we denote
the element of a matrix $X$ lying on the intersection of its $i$-th row
and $j$-th column. For integers $i, j$ we denote by $\delta _{i\,j}$ a
\emph{Kronecker delta} of $i$ and $j$, which is equal to $1$ if
$i = j$ and equal to $0$ otherwise.

For $A \in {\mathcal M}_{n\,k_{1}}({\mathbb F})$ and
$B \in {\mathcal M}_{n\,k_{2}}({\mathbb F})$ by
$A|B \in {\mathcal M}_{n\,(k_{1}+k_{2})}({\mathbb F})$ we denote a block
matrix defined by $A|B =
\begin{pmatrix}
A & B%
\end{pmatrix}
$.

For $A \in {\mathcal M}_{n\,k}({\mathbb F})$ by
$A^{t} \in {\mathcal M}_{k\,n}({\mathbb F})$ we denote a transpose of the
matrix $A$, i.e. $A^{t}_{i\,j} = A_{j\,i}$ for all $1 \le i \le k$,
$1 \le j \le n$.

We use the notation for submatrices following~\cite{Minc1984} and~\cite{Pierce1979}.
That is, by $A[J_{1}|J_{2}]$ we denote the $|J_{1}| \times |J_{2}|$ submatrix
of $A$ lying on the intersection of rows with the indices from
$J_{1}$ and the columns with the indices from $J_{2}$. By
$A(J_{1}|J_{2})$ we denote a submatrix of $A$ derived from it by striking
out from it the rows with indices belonging to $J_{1}$ and the columns
with the indices belonging to $J_{2}$. If one of the two index sets is
absent, then it means an empty set, i.e. $A(J_{1}|)$ denotes a matrix derived
from $A$ by striking out from it the rows with indices belonging to
$J_{1}$. We may skip curly brackets, i.e.
$A[1,2|3,4] = A[\{1,2\}|\{3,4\}]$. The notation with mixed brackets is
also used, i.e. $A(|1]$ denotes the first column of the matrix $A$. The
above notations are also used for vectors as well. In this case vectors
are considered as $n\times 1$ or $1 \times n$ matrices.

Let us start by introducing the notion of Cullis determinant. We provide
the basic definitions following~\cite{NAKAGAMI2007422}.

\begin{definition}
\label{defn1.2}
By $[n]$ we denote the set $\{1, \ldots , n\}$.
\end{definition}
%
\begin{definition}
\label{defn1.3}
By $\mathcal C^{X}_{k}$ we denote the set of injections from $[k]$ to
$X$.
\end{definition}
%
\begin{definition}%
\label{defn1.4}
By $\binom{X}{k}$ we denote the set of the images of injections from
$[k]$ to $X$, i.e. the set of all subsets of $X$ of cardinality $k$.
\end{definition}

\begin{definition}%
\label{defn1.5}
Let $c \in \binom{[n]}{k}$ be equal to $\{i_{1},\ldots , i_{k}\}$, where
$i_{1} < i_{2} < \ldots < i_{k}$, and $1 \le \alpha \le k$ be a natural
number. Then $c(\alpha )$ is defined by
\begin{equation*}
c(\alpha ) = i_{\alpha}.
\end{equation*}
\end{definition}

\begin{definition}%
\label{defn1.6}
Given a set $c \in \binom{[n]}{k}$ we denote by
$\sgn (c) = \sgn _{[n]} (c)$ the number
\begin{equation*}
(-1)^{\sum _{\alpha = 1}^{k} (c(\alpha ) - \alpha )}.
\end{equation*}
\end{definition}

\begin{note}%
\label{note1.7}
$\sgn _{[n]}(c)$ depends only on $c$ and does not depend on $n$.
\end{note}

\begin{definition}%
\label{defn1.8}
Given an injection $\sigma \in \mathcal C^{[n]}_{k}$ we denote by
$\sgn _{n\,k} (\sigma )$ the product
$\sgn (\pi _{\sigma}) \cdot \sgn _{[n]} (c)$, where
$\sgn (\pi _{\sigma})$ is the sign of the permutation
\begin{equation*}
\pi _{\sigma }=
\begin{pmatrix}
i_{1} & \ldots & i_{k}
\\
\sigma (1) & \ldots & \sigma (k)
\end{pmatrix}
,
\end{equation*}
$\{i_{1}, \ldots , i_{k}\} = \sigma ([k])$ and
$i_{1}<i_{2}<\ldots <i_{k}$.
\end{definition}

We omit the subscripts for $\sgn _{[n]}$ and $\sgn _{n\,k}$ if this cannot
lead to a misunderstanding.%

\begin{definition}[\cite{NAKAGAMI2007422}, Theorem 13]%
\label{def:CullisDet}
Let $n \ge k$, $X \in {\mathcal M}_{n\,k} (\mathbb F)$. Then the Cullis'
determinant $\det _{n\,k}(X)$ of $X$ is defined to be the function:
\begin{equation*}
\det _{n\,k} (X) = \sum _{\sigma \in \mathcal C^{[n]}_{k}} \sgn _{n\,k}
(\sigma ) x_{\sigma (1)\,1}x_{\sigma (2)\,2}\ldots x_{\sigma (k)\, k}.
\end{equation*}
We also denote $\det _{n\,k} (X)$ as follows
\begin{equation*}
\det _{n\,k}(X)=
\begin{vmatrix}
x_{1\,1} & \cdots & x_{1\,k}
\\
\vdots & \cdots & \vdots
\\
x_{n\,1} & \cdots & x_{n\,k}
\end{vmatrix}
_{n\,k}.
\end{equation*}
If $n = k$, then we also write $\det _{k}$ or $\det $ instead of
$\det _{n\,k}$ because in this case $\det _{n\,k}$ is clearly equal to
an ordinary determinant of a square matrix.
\end{definition}

Now it is possible to formulate the main theorem of this paper.

\begin{theorem}[{Theorem~\ref{thm:MainTheoremEvenKGe4}}]%
\label{MainTheoremEven}%
Assume that $|{\mathbb F}| > k \ge 4$, $n \ge k + 2$ and $n + k$ is even.
Let
$T\colon {\mathcal M}_{n\,k} ({\mathbb F}) \to {\mathcal M}_{n\,k} ({
\mathbb F})$ be a linear map. Then
$\det _{n\,k} (T(X)) = \det _{n\,k}(X)$ for all
$X \in {\mathcal M}_{n\,k} ({\mathbb F})$ if and only if there exist
$A \in {\mathcal M}_{n\,n}({\mathbb F})$ and
$B \in {\mathcal M}_{k\,k}({\mathbb F})$ 
such that $T(X) = AXB$ for all $X \in {\mathcal M}_{n\, k} ({\mathbb F})$ and
%
\begin{equation}
\label{thm:MainTheoremEven:eq1}
\det _{n\, k} \Bigl(A(|i_{1},\ldots , i_{k}]\Bigr) \cdot \det _{k}
\Bigl(B\Bigr) = (-1)^{\sum_{\alpha=1}^{k} (i_{\alpha} -\alpha)}\;\;\mbox{for all}\;\; 1 \le i_{1} < \ldots < i_{k} \le n.
\end{equation}
\end{theorem}

Our proof is based on the following observation: if $P$ is a polynomial
function of matrix coefficients, then for the given matrices $A, B$ the
expression $P(A + \lambda B)$ is a polynomial function of
$\lambda $, and for a certain $P$ it is possible to find the relationship
between the value of $\max _{A} (\deg _{\lambda }(P(A + \lambda B)))$ and
various properties of $B$.

Thus, assuming that $P$ is equal to Cullis' determinant, for a given matrix
$B$ we provide the relationship between the condition
$\max _{A} (\deg _{\lambda }P(A + \lambda B)) \le 1$ and the condition
$\rk (B) \le 1$ described in Lemma~\ref{lem:Rank1CullisDeg1} and Lemma~\ref{lem:CullisDeg1Rank1}
below. Here we need to assume the ground field is large enough because
we need to identify the polynomial with its coefficients.

We conclude from the obtained relationship that every linear map preserving
the Cullis' determinant must preserve the set of matrices of rank one because
in this case every linear map preserving the Cullis' determinant preserves
the value $\max _{A} (\deg _{\lambda }P(A + \lambda B))$. The description
of linear maps preserving matrices of rank one was obtained by Westwick~\cite{westwick1967}.
His result requires the bijectivity of the considered linear map. To establish
that every considered linear map is bijective we use the notion of radical
of a function and its properties which are introduced and studied by Waterhouse~\cite{Waterhouse1983}.

Then~\cite[Theorem~3.5]{westwick1967} implies that, under our assumptions,
every linear map preserving the Cullis' determinant has the form
$X \mapsto AXB$. Finally, we provide the necessary and sufficient condition
for a map having the such form to be a linear Cullis' determinant preserver which
is described in Lemma~\ref{lem:TwoSidedMulPreservesDet}.%
\bigskip

The paper is organized as follows: in Section~\ref{sec:prelim} we provide
the basic facts regarding $\det _{n\,k}$; in Section~\ref{sec:linDetNKpresAXB}
we introduce the criterion for a linear map of the form
$X\mapsto AXB$ to be a linear map preserving the Cullis' determinant; in
Section~\ref{sec:kernels} we prove that if $n + k$ is even $k \ge 2$ and $|{\mathbb F}| > k$, then every
linear map preserving $\det _{n\,k}$ is invertible by means of finding
the radical of $\det _{n\,k}$; Section~\ref{sec:maintheorem} is devoted
to proof of the main theorem which is based on the relationship between
the condition $\max _{A} (\deg _{\lambda }P(A + \lambda B)) \le 1$ and
the condition $\rk (B) \le 1$ (this relationship is also established in
Section~\ref{sec:maintheorem}); in Section~\ref{sec:K1} we provide the
characterization of linear maps preserving $\det _{n\,1}$; in Section~\ref{sec:K2}
we provide an example of linear map preserving $\det _{n\,2}$ for even
$n$ which does not admit a representation~\eqref{thm:MainTheoremEven:eq1}
and shows that the main theorem does not hold for $k = 2$ and even
$n$.

\section{Preliminaries}
\label{sec:prelim}

Let us first list the properties of $\det _{n\,k}$ which are similar to
corresponding properties of the ordinary determinant (see
\cite[\S 5, \S 27, \S 32]{cullis1913} or~\cite{NAKAGAMI2007422}
for detailed proofs).\eject

\begin{theorem}[{\cite[Theorem 13, Theorem 16]{NAKAGAMI2007422}}]
\label{thm2.1}
\begin{enumerate}
\item For $X \in {\mathcal M}_{n}(\mathbb F)$,
$\det _{n\,n}(X) = \det (X)$.
\item For $X \in {\mathcal M}_{n\,k}(\mathbb F)$, $\det _{n\,k}(X)$ is
a multilinear function of columns of $X$.
\item If a matrix $X \in {\mathcal M}_{n\,k}(\mathbb F)$ has two identical
columns or one of its columns is a linear combination of other columns,
then $\det _{n\,k}(X)$ is equal to zero.
\item For $X \in {\mathcal M}_{n\,k}(\mathbb F)$, interchanging any two
columns of $X$ changes the sign of $\det _{n\,k}(X)$.
\item Adding a linear combination of columns of $X$ to another column of
$X$ does not change $\det _{n\,k}(X)$.
\item For $X \in {\mathcal M}_{n\,k}(\mathbb F)$, $\det _{n\,k}(X)$ can
be calculated using the Laplace expansion along a column of $X$ (see Lemma~\ref{lem:DetNKLaplaceExp}
for precise formulation).
\end{enumerate}
\end{theorem}

\begin{corollary}%
\label{cor:CullisBinomialExpansion}%
Let $n \ge k$, $A, B \in {\mathcal M}_{n\,k} ({\mathbb F})$. Then
%
\begin{multline}
\label{eq:CullisBinomialExpansion}
\det _{n\,k} (A + \lambda B)
\\
= \sum _{d = 0}^{k}\lambda ^{d} \left ( \sum _{1 \le i_{1} < \ldots < i_{d}
\le k} \det _{n\,k}\Bigl(A(|1]\Big|\ldots \Big| B(|i_{1}] \Big|
\ldots \Big| B(|i_{d}] \Big| \ldots \Big| A(|k] \Bigr)\right ),
\end{multline}
where the both sides of the equality are considered as formal polynomials
in $\lambda $, i.e. as elements of ${\mathbb F}[\lambda ]$.
\end{corollary}
\begin{proof}%
This is a direct consequence of the multilinearity of
$\det _{n\,k}$ with respect to the columns of a matrix.
\end{proof}
%
\begin{corollary}%
\label{cor:DegDetABLEQK}%
If $A, B \in {\mathcal M}_{n\,k}({\mathbb F})$, then
$\deg _{\lambda}\left (\det _{n\,k} (A + \lambda B)\right ) \le k$.
\end{corollary}

\begin{definition}%
\label{defn2.4}
Let $c \in \binom{[n]}{k}$ and $c = \{i_{1}, \ldots , i_{k}\}$. By
$P_{c} \in {\mathcal M}_{n\,n}({\mathbb F})$ we denote a matrix defined
by
\begin{equation*}
P_{c} = E_{i_{1}\,i_{1}} + \ldots + E_{i_{k}\,i_{k}}.
\end{equation*}
\end{definition}

\begin{lemma}[{Cf.~\cite[Lemma~8]{NAKAGAMI2007422}}]
\label{lem2.5}
Let $X \in {\mathcal M}_{n\,k}({\mathbb F})$. Then
\begin{equation*}
\sgn _{[n]}(c) \det _{n\,k}(P_{c}X) = \det _{k} \Bigl(X[c|)\Bigr)\;\;\mbox{and}\;\;\det _{n\,k}(X) = \sum _{c \in \binom{[n]}{k}} \det _{n\,k} (P_{c} X).
\end{equation*}
\end{lemma}

\begin{corollary}%
\label{lem:DetNKAsSumProj}%
\label{cor:CullisAsSumDet}
The Cullis' determinant of a matrix
$X\in {\mathcal M}_{n\, k}({\mathbb F})$ is an alternating sum of basic
minors of $X$. That is,
%
\begin{equation}
\label{cor:CullisAsSumDet:eq}
\det _{n\,k}(X) = \sum _{c \in \binom{[n]}{k}}\sgn _{[n]}(c) \det _{k}
\Bigl(X[c|)\Bigr).
\end{equation}
\end{corollary}

\begin{lemma}[{Cf.~\cite[Proposition~24]{NAKAGAMI2007422}}]%
\label{lem:DetNKOfMatrixProduct}
Let $1 \le l \le k \le n$. Then 
\begin{equation*}
\det _{n\,l}(XY) = \sum _{d \in \binom{[k]}{l}} \det _{n\,l}\Bigl(X(|d]
\Bigr)\cdot \det _{l}\Bigl(Y[d|)\Bigr)\;\;\mbox{for all}\;\;X \in {\mathcal M}_{n\,k}({\mathbb F}),\,Y \in {\mathcal M}_{k\,l}({\mathbb F}).
\end{equation*}
\end{lemma}

\begin{corollary}%
\label{cor:rightmatrixmult}%
If $X \in {\mathcal M}_{n\,k}({\mathbb F})$ and
$Y \in {\mathcal M}_{k\,k}({\mathbb F})$, then
\begin{equation*}
\det _{n\,k}(XY) = \det _{n\,k}(X)\det _{k}(Y).
\end{equation*}
\end{corollary}

\begin{lemma}[{Cf.~\cite[Theorem~16]{NAKAGAMI2007422}}]%
\label{lem:DetNKLaplaceExp}
Let $1 < k \le n$. For any $n \times k$ matrix $X = (x_{i\,j})$ the expansion
of $\det _{n\,k}(X)$ along the $j$-th column is given by
\begin{equation*}
\det _{n\,k}(X) = \sum _{i = 1}^{n} (-1)^{i+j} x_{i\,j} \det _{(n-1)
\,(k-1)}\Bigl(X(i|j)\Bigr).
\end{equation*}
\end{lemma}

\begin{lemma}[Invariance of $\det _{n\,k}$ under semi-cyclic shifts, cf.~{\cite[Theorem 3.6]{amiri2010}}]%
\label{lem:PermKNEvenSemicyclic}%
If $k \le n$, and $k + n$ is even, then for all
$X = (x_{i\,j}) \in {\mathcal M}_{n\,k}({\mathbb F})$ and
$i \in \{1, \ldots , n\}$
\begin{equation*}
(-1)^{(n-i)k}
\begin{vmatrix}
x_{i\,1} & \cdots & x_{i\,k}
\\
\vdots & \ddots & \vdots
\\
x_{n\,1} & \cdots & x_{n\,k}
\\
-x_{1\,1} & \cdots & -x_{1\,k}
\\
\vdots & \ddots & \vdots
\\
-x_{(i-1)\,1} & \cdots & -x_{(i-1)\,k}
\end{vmatrix}
_{n\,k} =
\begin{vmatrix}
x_{1\,1} & \cdots & x_{1\,k}
\\
\vdots & \ddots & \vdots
\\
x_{n\,1} & \cdots & x_{n\,k}
\end{vmatrix}
_{n\,k}.
\end{equation*}
Here the matrix on the left-hand side of the equality is obtained from
$X$ by performing the following sequence of operations: the row cyclical
shift sending the $i$-th row of $X$ to the first row of the result; multiplying
the bottom $i-1$ rows by $-1$.
\end{lemma}

\begin{lemma}[Cf.~{\cite[Lemma~20]{NAKAGAMI2007422}}]%
\label{lem:NAKAGAMILEMMA20} %
Assume that $n > k \ge 1$. Let
$X \in {\mathcal M}_{n\,k}({\mathbb F})$ and
$Y \in {\mathcal M}_{n\,(k+1)}$ be defined by
$Y = X |
\begin{psmallmatrix}
1
\\
\vdots
\\
1%
\end{psmallmatrix}
$. Then
\begin{equation*}
\det _{n\,(k+1)}(Y) =
\begin{cases}
\det _{n\,k}(X), & \text{$n + k$ is odd},
\\
0, & \text{$n + k$ is even}.
\end{cases}
\end{equation*}
\end{lemma}

Recall the following general facts regarding polynomials over a field.

\begin{lemma}[{Cf.~\cite[IV, \S 1, Corollary 1.8]{Lang2002}}]%
\label{lem2.12}
Let ${\mathbb F}$ be a finite field with $q$ elements. Let $f$ be a polynomial
in $n$ variables over ${\mathbb F}$ such that the degree of $f$ in each
variable is less than $q$. If $f$ defines the zero function on
${\mathbb F}^{n}$, then $f = 0$.
\end{lemma}

\begin{corollary}%
\label{DegFGLessEqual}%
Let $f, g$ be polynomials in $n$ variables over ${\mathbb F}$ such that
either ${\mathbb F}$ is infinite or the degree of $f$ and $g$ in each variable
is less than $|{\mathbb F}|$. If $f$ and $g$ define the same function on
${\mathbb F}^{n}$, then $f = g$.
\end{corollary}

In order to prove that every linear map preserving $\det _{n\,k}$ is invertible
we use the notion and the properties of radical of a function on a finite-dimensional vector
space following~\cite{Waterhouse1983}.

\begin{definition}[Cf.~{\cite[text at the beginning of Section~1]{Waterhouse1983}}]%
\label{def:radical}%
Let ${\mathbb F}$ be a field, $V$ be a finite-dimensional vector space
over ${\mathbb F}$, and let $f$ be a function from $V$ to
${\mathbb F}$. The \emph{radical} of $f$, denoted by $\rad (f)$, is a subset
of $V$ defined by
\begin{equation*}
\rad (f) = \{\mathbf w \mid f(\mathbf v + \lambda \mathbf w) = f(
\mathbf v)\;\;\text{for all}\;\; \mathbf v \in V,\;\; \lambda \in {
\mathbb F}\}.
\end{equation*}
\end{definition}

\begin{lemma}[Cf.~{\cite[Proposition~1]{Waterhouse1983}}]%
\label{lem:RadicalCriterion}
$\rad (f)$ is nonzero if and only if there is a noninvertible linear map
$T\colon V\to V$ satisfying $f(T\mathbf v) = f(\mathbf v)$ for all
$\mathbf v$ in $V$.
\end{lemma}

\section{Linear $det_{n\,k}$-preservers of the form $X \mapsto AXB$}
\label{sec:linDetNKpresAXB}

The common type of $\det _{n\,k}$-preservers can be obtained by the two-sided
matrix multiplication as it follows from Lemma~\ref{lem:TwoSidedMulPreservesDet}
below.

\begin{lemma}%
\label{lem:TwoSidedMulPreservesDet}%
Let $n \ge k$, $A \in {\mathcal M}_{n\,n}({\mathbb F})$,
$B \in {\mathcal M}_{k\,k}({\mathbb F})$, and let
$T\colon {\mathcal M}_{n\,k}({\mathbb F}) \to {\mathcal M}_{n\,k}({
\mathbb F})$ be the linear map defined by
%
\begin{equation}
\label{eq:TwoSideMul}
T(X) = AXB\;\;\mbox{for all}\;\;X \in {\mathcal M}_{n\,k}({\mathbb F})
\end{equation}
Then $\det _{n\,k} (T(X)) = \det _{n\,k}(X)$ for all $X \in {\mathcal M}_{n\,k} ({\mathbb F})$ if and only if
%
\begin{equation}
\label{eq:TwoSideMul2}
\det _{n\,k} \Bigl(A(|d]\Bigr)\cdot \det _{k} \Bigl(B\Bigr) = \sgn (d)\;\;\mbox{for all}\;\;d \in \binom{[n]}{k}.
\end{equation}

\end{lemma}
\begin{proof}
Assume that
$A \in {\mathcal M}_{n\,n}({\mathbb F}), B \in {\mathcal M}_{k\,k}({
\mathbb F})$ are the matrices satisfying the condition~\eqref{eq:TwoSideMul2}. Let $X \in {\mathcal M}_{n\,k}({
\mathbb F})$.
Then
%
\begin{equation}
\label{TwoSideMul:eq:1}
\det _{n\,k} (T(X)) = \det _{n\,k}(AXB) = \det _{n\,k}((AX)B).
\end{equation}
Since the triple $(k,k,n)$ satisfies the inequalities
$1 \le k \le k \le n$ and
$AX \in {\mathcal M}_{n\,k}({\mathbb F}), B \in {\mathcal M}_{k\,k}({
\mathbb F})$, then
%
\begin{equation}
\label{TwoSideMul:eq:2}
\det _{n\,k}((AX)B) = \det _{n\,k}(AX)\det _{k}(B)
\end{equation}
by Corollary~\ref{cor:rightmatrixmult}. Similarly, since the triple
$(k,n,n)$ satisfies the inequalities $1 \le k \le n \le n$ and
$A \in {\mathcal M}_{n\,n}({\mathbb F}), X \in {\mathcal M}_{n\,k}({
\mathbb F})$, then by Lemma~\ref{lem:DetNKOfMatrixProduct} we have
%
\begin{equation}
\label{TwoSideMul:eq:3}
\det _{n\,k}(AX)\det _{k} (B) = \left (\sum _{d \in \binom{[n]}{k}}
\det _{n\,k}\Bigl(A(|d]\Bigr)\det _{k}\Bigl(X[d|)\Bigr)\right ) \det _{k}
\Bigl(B\Bigr).
\end{equation}
Hence,
%
\begin{multline}
\label{TwoSideMul:eq:4}
\left (\sum _{d \in \binom{[n]}{k}} \det _{n\,k}\Bigl(A(|d]\Bigr)
\det _{k}\Bigl(X[d|)\Bigr)\right ) \det _{k}\Bigl(B\Bigr)
\\
= \sum _{d \in \binom{[n]}{k}} \left (\det _{n\,k}\Bigl(A(|d]\Bigr)
\det _{k}\Bigl(B\Bigr)\right ) \cdot \det _{k}\Bigl(X[d|)\Bigr).
\end{multline}
The condition of the lemma implies that
%
\begin{equation}
\label{TwoSideMul:eq:5}
\sum _{d \in \binom{[n]}{k}} \left (\det _{n\,k}\Bigl(A(|d])\Bigr)
\det _{k}(B)\right ) \cdot \det _{k}\Bigl(X[d|)\Bigr)
= \sum _{d \in \binom{[n]}{k}} \left ( \sgn _{[n]}(d)\right ) \cdot
\det _{k}\Bigl(X[d|)\Bigr).
\end{equation}
By Corollary~\ref{lem:DetNKAsSumProj},
%
\begin{equation}
\label{TwoSideMul:eq:7}
\sum _{d \in \binom{[n]}{k}} \sgn _{[n]}(d)\cdot \det _{k}\Bigl(X[d|)
\Bigr) = \det _{n\,k} (X).
\end{equation}
Thus, by aligning the equalities~\eqref{TwoSideMul:eq:1}--\eqref{TwoSideMul:eq:7}
together, we obtain that
\begin{equation*}
\det _{n\,k} (T(X)) = \det _{n\,k} (X).
\end{equation*}

To prove the converse, assume that
$A \in {\mathcal M}_{n\,n}({\mathbb F})$,
$B \in {\mathcal M}_{k\,k}({\mathbb F})$ are such that
$\det _{n\,k}(AXB) = \det _{n\,k}(X)$ for all
$X \in {\mathcal M}_{n\,k}({\mathbb F})$. Let $d \in \binom{[n]}{k}$ and
$X^{(d)}$ be defined by
$ X^{(d)} = \sum \limits _{j=1}^{k} E_{i_{j}\,j} $, where
$1 \le i_{1} < \ldots < i_{k} \le n$ and
$\{i_{1},\ldots , i_{k}\} = d$.

Consider the expansion of $\det _{n\,k}(X^{(d)})$ by the definition of
the Cullis' determinant.
\begin{equation*}
\det _{n\,k}(X^{(d)}) = \sum _{\sigma \in \mathcal C^{[n]}_{k}} \sgn _{n
\,k} (\sigma ) \cdot x^{(d)}_{\sigma (1)\,1}\cdot x^{(d)}_{\sigma (2)
\,2} \cdot \ldots \cdot x^{(d)}_{\sigma (k)\,k}.
\end{equation*}
On the one hand,
\begin{equation*}
\det _{n\,k}(X^{(d)}) = \sgn _{n\,k} (\sigma _{0}) \cdot 1 \cdot 1
\cdot \ldots \cdot 1 = \sgn _{n\,k} (\sigma _{0}),
\end{equation*}
where $\sigma _{0}(j) = i_{j}$ for all $1 \le j \le k$, because
$x^{(d)}_{\sigma (1)\,1}\cdot x^{(d)}_{\sigma (2)\,2} \cdot \ldots
\cdot x^{(d)}_{\sigma (k)\,k} \neq 0$ if and only if
$\sigma (j) = i_{j}$ for all $1 \le j \le k$. The definition of
$\sgn _{n\,k}$ implies that
\begin{equation*}
\sgn _{n\,k} (\sigma _{0}) = \sgn _{[n]}(\{i_{1}, \ldots , i_{k}\}) =
\sgn _{[n]}(d).
\end{equation*}
Thus,
%
\begin{equation}
\label{TwoSideMul:eq:8}
\det _{n\,k}(X^{(d)}) = \sgn _{[n]}(d).
\end{equation}

On the other hand, it follows from the definition of $X^{(d)}$ that
\begin{equation*}
AX^{(d)}B = \Bigl(AX^{(d)}\Bigr)B = \Bigl(A(|d]\Bigr)B
\end{equation*}
and consequently
%
\begin{equation}
\label{TwoSideMul:eq:9}
\det _{n\,k}\Bigl(AX^{(d)}B\Bigr) = \det _{n\,k} \Bigl(A(|d]B\Bigr) =
\det _{n\,k} \Bigl(A(|d]\Bigr)\det _{k} (B)
\end{equation}
by Corollary~\ref{cor:rightmatrixmult}. Since we assumed that
$\det _{n\,k}(AXB) = \det _{n\,k}(X)$ for all
$X \in {\mathcal M}_{n\,k}({\mathbb F})$, then we conclude that the equalities~\eqref{TwoSideMul:eq:8}
and~\eqref{TwoSideMul:eq:9} could be aligned together which implies that
%
\begin{equation}
\label{TwoSideMul:eq:10}
\det _{n\,k} \Bigl(A(|d]\Bigr)\det _{k} (B) = \det _{n\,k}(AX^{(d)}B) =
\det _{n\,k}(X^{(d)}) = \sgn _{[n]}(d).
\end{equation}
This means that $A$ and $B$ satisfy the condition~\eqref{eq:TwoSideMul2}
because the equality~\eqref{TwoSideMul:eq:10} holds for all
$d \in \binom{[n]}{k}$.
\end{proof}

Thus, Lemma~\ref{lem:TwoSidedMulPreservesDet} provides a criterion for
a linear map of the form $X \mapsto AXB$ to preserve $\det _{n\,k}$. However,
as one can see below, in certain cases it may be more convenient to check
it directly. For example, in Lemma~\ref{lem:RSwitchPreservesDetNK} we
prove that the linear map ${\mathbf R}_{i}$ introduced in Definition~\ref{def:RSwitchDef}
preserves $\det _{n\,k}$ by expanding the corresponding Cullis' determinant.
The map ${\mathbf R}_{i}$ has the form $X \mapsto AX$ and is used in Section~\ref{sec:maintheorem}.

\begin{definition}%
\label{def:RSwitchDef}%
Let $n, k \ge 1$ and $1 \le i < n$. By
${\mathbf R}_{i} \colon {\mathcal M}_{n\,k}({\mathbb F}) \to {
\mathcal M}_{n\,k}({\mathbb F})$ we denote a linear map defined by
\begin{equation*}
\label{def:RSwitchDef:eqq}
{\mathbf R}_{i}(
\begin{pmatrix}
x_{1\,1} & \cdots & x_{1\,k}
\\
\vdots & \ddots & \vdots
\\
x_{n\,1} & \cdots & x_{n\,k}
\end{pmatrix})
=
\begin{pmatrix}
x_{1\,1} & \cdots & x_{1\,k}
\\
\vdots & \ddots & \vdots
\\
x_{(i+1)\, 1} & \cdots & x_{(i+1)\, k}
\\
2x_{(i+1)\, 1} - x_{i\,1} & \cdots & 2x_{(i+1)\, k} - x_{i\,k}
\\
\vdots & \ddots & \vdots
\\
x_{n\,1} & \cdots & x_{n\,k}
\end{pmatrix}
\end{equation*}
for all $X = (x_{i\,j}) \in {\mathcal M}_{n\,k}({\mathbb F})$. That is,
${\mathbf R}_{i}(X)$ is obtained from $X$ by performing the following sequence
of operations:
\begin{enumerate}
\item exchanging the $i$-th row and the $(i+1)$-th row of $X$;
\item multiplying the $(i+1)$-th row by $(-1)$;
\item adding the $i$-th row multiplied by $2$ to the $(i+1)$-th row.
\end{enumerate}
\end{definition}

\begin{remark}%
\label{rem3.3}
In Definition~\ref{def:RSwitchDef} we assume that $2 = 1 + 1$. Thus, if
$\Char ({\mathbb F}) = 2$, then $2 = 0$ and ${\mathbf R}_{i}(X)$ is obtained
from $X$ by exchanging its $i$-th row and $(i+1)$-th row.
\end{remark}

\begin{lemma}%
\label{lem:RSwitchCommutesWithJoin}%
Let $k, m, n \ge 1 $, $A \in {\mathcal M}_{n\,k}({\mathbb F})$,
$B \in {\mathcal M}_{n\,m}({\mathbb F})$ and $1 \le i < n$. Then
${\mathbf R}_{i}(A|B) = {\mathbf R}_{i}(A)|{\mathbf R}_{i}(B)$.
\end{lemma}
\begin{proof}%
Indeed, ${\mathbf R}_{i}$ is a composition of row operations by Definition~\ref{def:RSwitchDef}.
Therefore, it commutes with the column operation
$A, B \mapsto A | B$. This implies the required equality.
\end{proof}

\begin{lemma}%
\label{lem:RSwitchPreservesDetNK}%
Assume that $n \ge k$. Then ${\mathbf R}_{i}$ is an invertible linear map
preserving $\det _{n\,k}$ for all $1 \le i < n$.
\end{lemma}
\begin{proof}%
The invertibility of ${\mathbf R}_{i}$ follows directly from its definition.
Let us show that ${\mathbf R}_{i}$ preserves $\det _{n\,k}$.

Let $X \in {\mathcal M}_{n\,k}({\mathbb F})$ and
$Y = {\mathbf R}_{i}(X)$. Let us express $\det _{n\,k}(X)$ and
$\det _{n\,k}(Y)$ as a sum of minors applying Corollary~\ref{lem:DetNKAsSumProj}.
We obtain that
%
\begin{equation}
\label{lem:RSwitchPreservesDetNK:eq1}
\det _{n\,k}(X) = \sum _{c \in \binom{[n]}{k}}\sgn _{[n]}(c) \det _{k}
\Bigl(X[c|)\Bigr) = A + B + C + D,
\end{equation}
where
\begin{gather*}
A = \sum _{\substack{c \in \binom{[n]}{k}\\ i, i+1 \not \in c}}\sgn _{[n]}(c)
\det _{k}\Bigl(X[c|)\Bigr),\qquad B = \sum _{
\substack{c \in \binom{[n]}{k}\\ i \in c, i + 1 \not \in c}}\sgn _{[n]}(c)
\det _{k}\Bigl(X[c|)\Bigr),
\\
C = \sum _{
\substack{c \in \binom{[n]}{k}\\ i \not \in c, i + 1 \in c}}\sgn _{[n]}(c)
\det _{k}\Bigl(X[c|)\Bigr), \qquad D = \sum _{
\substack{c \in \binom{[n]}{k}\\ i, i+1 \in c}}\sgn _{[n]}(c) \det _{k}
\Bigl(X[c|)\Bigr).
\end{gather*}
Similarly,
%
\begin{equation}
\label{lem:RSwitchPreservesDetNK:eq2}
\det _{n\,k}(Y) = \sum _{c \in \binom{[n]}{k}}\sgn _{[n]}(c) \det _{k}
\Bigl(Y[c|)\Bigr) = E + F + G + H,
\end{equation}
where
\begin{gather*}
E = \sum _{\substack{c \in \binom{[n]}{k}\\ i, i+1 \not \in c}}\sgn _{[n]}(c)
\det _{k}\Bigl(Y[c|)\Bigr),\qquad F = \sum _{
\substack{c \in \binom{[n]}{k}\\ i \in c, i + 1 \not \in c}}\sgn _{[n]}(c)
\det _{k}\Bigl(Y[c|)\Bigr),
\\
G = \sum _{
\substack{c \in \binom{[n]}{k}\\ i \not \in c, i + 1 \in c}}\sgn _{[n]}(c)
\det _{k}\Bigl(Y[c|)\Bigr), \qquad H = \sum _{
\substack{c \in \binom{[n]}{k}\\ i, i+1 \in c}}\sgn _{[n]}(c) \det _{k}
\Bigl(Y[c|)\Bigr).
\end{gather*}
The linear map ${\mathbf R}_{i}$ affects only the $i$-th row and the
$(i+1)$-th row of $X$. Hence, the definition of $Y$ implies that
%
\begin{equation}
\label{lem:RSwitchPreservesDetNK:eq3}
E = A.
\end{equation}

Similarly, the definition of $Y$, the multilinearity and the skew-symmetricity
of $\det _{k}$ along the matrix rows imply that
%
\begin{equation}
\label{lem:RSwitchPreservesDetNK:eq4}
H = D.
\end{equation}

Let $c \in \binom{[n]}{k}$ be such that $i \in c$ and
$i + 1 \not \in c$. Then $c' = c \triangle \{i, i+1\}$ is such that
$i \not \in c'$ and $i + 1 \in c'$. Moreover,
$\sgn _{[n]}(c) = -\sgn _{[n]}(c')$ by the definition of $c'$. The definition
of $Y$ implies that
\begin{equation*}
\det _{k}\Bigl(Y[c|)\Bigr) = \det _{k}\Bigl(X[c'|)\Bigr).
\end{equation*}
Hence,
\begin{equation*}
\sgn _{[n]}(c) \det _{k}\Bigl(Y[c|)\Bigr) = -\sgn _{[n]}(c') \det _{k}
\Bigl(X[c'|)\Bigr)
\end{equation*}
for all $c \in \binom{[n]}{k}$ such that $i \in c$ and
$i + 1 \not \in c$. Therefore,
%
\begin{equation}
\label{lem:RSwitchPreservesDetNK:eq5}
F = -C.
\end{equation}

Let $c \in \binom{[n]}{k}$ be such that $i \not \in c$ and
$i + 1 \in c$. Then $c' = c \triangle \{i, i+1\}$ is such that
$i \in c'$ and $i + 1 \not \in c'$. As in the previous case,
$\sgn _{[n]}(c) = -\sgn _{[n]}(c')$. The definition of $Y$ and the multilinearity
of $\det _{k}$ imply that
\begin{equation*}
\det _{k}\Bigl(Y[c|)\Bigr) = -\det _{k}\Bigl(X[c'|)\Bigr) + 2\det _{k}
\Bigl(X[c|)\Bigr).
\end{equation*}
Hence,
\begin{equation*}
\sgn _{[n]}(c) \det _{k}\Bigl(Y[c|)\Bigr) = \sgn _{[n]}(c')\det _{k}
\Bigl(X[c'|)\Bigr) + 2\sgn _{[n]}(c)\det _{k}\Bigl(X[c|)\Bigr)
\end{equation*}
for all $c \in \binom{[n]}{k}$ such that $i \not \in c$ and
$i + 1 \in c$. Therefore,
%
\begin{equation}
\label{lem:RSwitchPreservesDetNK:eq6}
G = B + 2C.
\end{equation}

By substituting~\eqref{lem:RSwitchPreservesDetNK:eq3},
\eqref{lem:RSwitchPreservesDetNK:eq4},
\eqref{lem:RSwitchPreservesDetNK:eq5} and~\eqref{lem:RSwitchPreservesDetNK:eq6}
into~\eqref{lem:RSwitchPreservesDetNK:eq2} we obtain that
\begin{multline*}
\det _{n\,k}({\mathbf R}_{i}(X)) = \det _{n\,k}(Y)
\overset{\text{\eqref{lem:RSwitchPreservesDetNK:eq2}}}{=\joinrel=\joinrel=}
E + F + G + H
\\
= A + (-C) + (B + 2C) + D = A + B + C + D
\overset{\text{\eqref{lem:RSwitchPreservesDetNK:eq1}}}{\myjoin}
\det _{n\,k}(X).
\end{multline*}
This equality holds for all $X \in {\mathcal M}_{n\,k}({\mathbb F})$. Thus,
${\mathbf R}_{i}$ preserves $\det _{n\,k}$.
\end{proof}

\section{The radical of $\det _{n\,k}$ if $n + k$ is even}
\label{sec:kernels}

\begin{lemma}%
\label{lem:DetNKXiEqX1}%
Assume that $n \ge k \ge 2$ and $n+k$ is even. Let
$x_{1}, \ldots , x_{n} \in {\mathbb F}$ and
\begin{equation*}
X =
\begin{pmatrix}
x_{1} & 0 & 0 & \cdots & 0
\\
x_{2} & 1 & 0 &\cdots & 1
\\
x_{3} & 0 & 1 &\cdots & 1
\\
\vdots & \vdots & \vdots & \ddots & \vdots
\\
x_{k} & 0 & 0 & \cdots & 1
\\
\vdots & \vdots & \vdots & \ddots & \vdots
\\
x_{n} & 0 & 0 & \cdots & 1
\end{pmatrix}
\in {\mathcal M}_{n\,k}({\mathbb F}).
\end{equation*}
Then
\begin{equation*}
\det _{n\,k}(X) = x_{1}.
\end{equation*}
\end{lemma}
\begin{proof}
The multilinearity of the Cullis' determinant along the last column implies
that
%
\begin{equation}
\label{lem:DetNKXiEqX1:eq1}
\begin{vmatrix}
x_{1} & 0 & 0 & \cdots & 0
\\
x_{2} & 1 & 0 &\cdots & 1
\\
x_{3} & 0 & 1 &\cdots & 1
\\
\vdots & \vdots & \vdots & \ddots & \vdots
\\
x_{k} & 0 & 0 & \cdots & 1
\\
\vdots & \vdots & \vdots & \ddots & \vdots
\\
x_{n} & 0 & 0 & \cdots & 1
\end{vmatrix}
_{n\,k} =
\begin{vmatrix}
x_{1} & 0 & 0 & \cdots & 1
\\
x_{2} & 1 & 0 &\cdots & 1
\\
x_{3} & 0 & 1 &\cdots & 1
\\
\vdots & \vdots & \vdots & \ddots & \vdots
\\
x_{k} & 0 & 0 & \cdots & 1
\\
\vdots & \vdots & \vdots & \ddots & \vdots
\\
x_{n} & 0 & 0 & \cdots & 1
\end{vmatrix}
_{n\,k} -
\begin{vmatrix}
x_{1} & 0 & 0 & \cdots & 1
\\
x_{2} & 1 & 0 &\cdots & 0
\\
x_{3} & 0 & 1 &\cdots & 0
\\
\vdots & \vdots & \vdots & \ddots & \vdots
\\
x_{k} & 0 & 0 & \cdots & 0
\\
\vdots & \vdots & \vdots & \ddots & \vdots
\\
x_{n} & 0 & 0 & \cdots & 0
\end{vmatrix}
_{n\,k}.
\end{equation}

Consider the first term in the difference~\eqref{lem:DetNKXiEqX1:eq1}.
Lemma~\ref{lem:NAKAGAMILEMMA20} implies that
\begin{equation*}
\begin{vmatrix}
x_{1} & 0 & 0 & \cdots & 0 & 1
\\
x_{2} & 1 & 0 &\cdots & 0 & 1
\\
x_{3} & 0 & 1 &\cdots & 0 & 1
\\
\vdots & \vdots & \vdots & \ddots & \vdots & \vdots
\\
x_{k-1} & 0 & 0 & \cdots & 1 & 1
\\
x_{k} & 0 & 0 & \cdots & 0& 1
\\
\vdots & \vdots & \vdots & \ddots & \vdots & \vdots
\\
x_{n} & 0 & 0 & \cdots & 0 & 1
\end{vmatrix}
_{n\,k} =
\begin{vmatrix}
x_{1} & 0 & 0 & \cdots & 0
\\
x_{2} & 1 & 0 &\cdots & 0
\\
x_{3} & 0 & 1 &\cdots & 0
\\
\vdots & \vdots & \vdots & \ddots & \vdots
\\
x_{k-1} & 0 & 0 & \cdots & 1
\\
x_{k} & 0 & 0 & \cdots & 0
\\
\vdots & \vdots & \vdots & \ddots & \vdots
\\
x_{n} & 0 & 0 & \cdots & 0
\end{vmatrix}
_{n\,(k-1)}
\end{equation*}
because $n + (k - 1)$ is odd. Then, using the Laplace expansion along the
last column $k-2$ times we obtain that
%
\begin{equation}
\label{lem:DetNKXiEqX1:eq2}
\begin{vmatrix}
x_{1} & 0 & 0 & \cdots & 0
\\
x_{2} & 1 & 0 &\cdots & 0
\\
x_{3} & 0 & 1 &\cdots & 0
\\
\vdots & \vdots & \vdots & \ddots & \vdots
\\
x_{k-1} & 0 & 0 & \cdots & 1
\\
x_{k} & 0 & 0 & \cdots & 0
\\
\vdots & \vdots & \vdots & \ddots & \vdots
\\
x_{n} & 0 & 0 & \cdots & 0
\end{vmatrix}
_{n\,(k-1)} =
\begin{vmatrix}
x_{1}
\\
x_{k}
\\
\vdots
\\
x_{n}
\end{vmatrix}
_{(n-k+2)\,1} = x_{1} - \sum _{i = k}^{n} (-1)^{i-k}x_{i}.
\end{equation}

Next, consider the second term in the difference~\eqref{lem:DetNKXiEqX1:eq1}.
Using the Laplace expansion along the last column $k-1$ times we obtain
that
%
\begin{multline}
\label{lem:DetNKXiEqX1:eq3}
\begin{vmatrix}
x_{1} & 0 & 0 & \cdots & 1
\\
x_{2} & 1 & 0 &\cdots & 0
\\
x_{3} & 0 & 1 &\cdots & 0
\\
\vdots & \vdots & \vdots & \ddots & \vdots
\\
x_{k} & 0 & 0 & \cdots & 0
\\
\vdots & \vdots & \vdots & \ddots & \vdots
\\
x_{n} & 0 & 0 & \cdots & 0
\end{vmatrix}_{n\,k} = (-1)^{(1+k) + \left[((k-2) + (k-1)) + ((k-3) + (k-2)) + \ldots + (1 + 2)\right]}
\begin{vmatrix}
x_{k}
\\
\vdots
\\
x_{n}
\end{vmatrix}_{(n-k+1)\,1}
\\
= (-1)^{(1 + k) + [k - 2]}
\begin{vmatrix}
x_{k}
\\
\vdots
\\
x_{n}
\end{vmatrix}_{(n-k+1)\,1} = -\sum _{i = k}^{n} (-1)^{i-k}x_{i} 
\end{multline}
By substituting~\eqref{lem:DetNKXiEqX1:eq2} and~\eqref{lem:DetNKXiEqX1:eq3}
into~\eqref{lem:DetNKXiEqX1:eq1} we finally obtain that
\begin{equation*}
\det _{n\,k}(X) = x_{1} - \sum _{i = k}^{n} (-1)^{i-k}x_{i} - \left (-
\sum _{i = k}^{n} (-1)^{i-k}x_{i}\right ) = x_{1}.\qedhere
\end{equation*}
\end{proof}

\begin{lemma}%
\label{lem:NKEvenAdditionPreserverHasZero11}%
Assume that $n \ge k \ge 2$, $|{\mathbb F}| > k$ and $n + k$ is even. If
$Y = (y_{i\,j}) \in {\mathcal M}_{n\,k} ({\mathbb F})$ is such that
%
\begin{equation}
\label{lem:NKEvenAdditionPreserverHasZero11:cond}
\det _{n\,k} (A + \lambda Y) = \det _{n\,k} (A)\;\;\mbox{for all}\;\;A \in {\mathcal M}_{n\,k}({\mathbb F}),\;
\lambda \in {\mathbb F},
\end{equation}
 then $y_{1\,1} = 0$.
\end{lemma}
\begin{proof}
Assume that $Y$ satisfies the condition~\eqref{lem:NKEvenAdditionPreserverHasZero11:cond}.
Let $A \in {\mathcal M}_{n\,k}({\mathbb F})$ be defined by
\begin{equation*}
A = E_{2\,2} + E_{3\,3} + \ldots + E_{(k-1)\,(k-1)} + E_{2\,k} +
\ldots + E_{n\,k} =
\begin{pmatrix}
0 & 0 & 0 & \cdots & 0
\\
0 & 1 & 0 &\cdots & 1
\\
0 & 0 & 1 &\cdots & 1
\\
\vdots & \vdots & \vdots & \ddots & \vdots
\\
0 & 0 & 0 & \cdots & 1
\\
\vdots & \vdots & \vdots & \ddots & \vdots
\\
0 & 0 & 0 & \cdots & 1
\end{pmatrix}
.
\end{equation*}

Let $P \in {\mathbb F}[\lambda ]$ be a polynomial defined by
\begin{equation*}
P = \det _{n\, k} (A + \lambda Y) = a_{0} + a_{1}\lambda + \ldots + a_{k}
\lambda ^{k}.
\end{equation*}
On the one hand, the condition~\eqref{lem:NKEvenAdditionPreserverHasZero11:cond}
implies that
\begin{equation*}
a_{0} + a_{1}\lambda + \ldots + a_{k} \lambda ^{k} = \det _{n\,k} (A)
\;\;\text{for all}\;\; \lambda \in {\mathbb F}.
\end{equation*}
Therefore,
%
\begin{equation}
\label{lem:NKEvenAdditionPreserverHasZero11:eq3}
a_{1} = \ldots = a_{k} = 0
\end{equation}
by Corollary~\ref{DegFGLessEqual} because $|{\mathbb F}| > k$.

On the other hand, \eqref{eq:CullisBinomialExpansion} implies that
\begin{equation*}
a_{1} = \sum _{1 \le i_{1} \le k} \det _{n\,k}\Bigl(A(|1]\Big|\ldots
\Big| Y(|i_{1}] \Big| \ldots \Big| A(|k] \Bigr)
\end{equation*}
being a coefficient of $\lambda $ in expansion~\eqref{eq:CullisBinomialExpansion}.
Since $A(|1]$ is a zero column, we conclude by splitting off the first
term that
\begin{multline*}
a_{1} = \sum _{1 \le i_{1} \le k} \det _{n\,k}\Bigl(A(|1]\Big|\ldots
\Big| Y(|i_{1}] \Big| \ldots \Big| A(|k] \Bigr)
\\
= \det _{n\,k} \Bigl(Y(|1] \Big| A(|2] \Big| \ldots \Big| A(|k]\Bigr) +
\sum _{2 \le i_{1} \le k} \det _{n\,k}\Bigl(A(|1]\Big|\ldots \Big| Y(|i_{1}]
\Big| \ldots \Big| A(|k] \Bigr)
\\
= \det _{n\,k} \Bigl(Y(|1] \Big| A(|2] \Big| \ldots \Big| A(|k]\Bigr) +
\sum _{2 \le i_{1} \le k} \det _{n\,k}\Bigl(0\Big|\ldots \Big| Y(|i_{1}]
\Big| \ldots \Big| A(|k] \Bigr)
\\
= \det _{n\,k} \Bigl(Y(|1] \Big| A(|2] \Big| \ldots \Big| A(|k]\Bigr) +
\sum _{2 \le i_{1} \le k} 0
= \det _{n\,k} \Bigl(Y(|1] \Big| A(|2] \Big| \ldots \Big| A(|k]\Bigr).
\end{multline*}

By substituting the right-hand side of the above equality into~\eqref{lem:NKEvenAdditionPreserverHasZero11:eq3}
we obtain that
%
\begin{equation}
\label{lem:NKEvenAdditionPreserverHasZero11:eq1}
\det _{n\,k} \Bigl(Y(|1] \Big| A(|2] \Big| \ldots \Big| A(|k]\Bigr) = 0.
\end{equation}

Now note that $Y(|1] \Big| A(|2] \Big| \ldots \Big| A(|k]$ has the form
described in the statement of Lemma~\ref{lem:DetNKXiEqX1} for
$x_{1} = y_{1\,1}, \ldots , x_{n} = y_{n\,1}$. This implies that%
%
\begin{equation}
\label{lem:NKEvenAdditionPreserverHasZero11:eq2}
\det _{n\,k} \Bigl(Y(|1] \Big| A(|2] \Big| \ldots \Big| A(|k]\Bigr) = y_{1
\,1}.
\end{equation}

By substituting~\eqref{lem:NKEvenAdditionPreserverHasZero11:eq1} into~\eqref{lem:NKEvenAdditionPreserverHasZero11:eq2}
we obtain that
\begin{equation*}
y_{1\,1} = \det _{n\, k} \Bigl(Y(|1] \Big| A(|2] \Big| \ldots \Big| A(|k]
\Bigr) = 0.\qedhere
\end{equation*}
\end{proof}

\begin{definition}%
\label{def:SCSDef}%
Let $n, k \ge 1$, $1 \le i \le n$ and $1 \le j \le k$. By
${\mathbf S}_{i\,j} \colon {\mathcal M}_{n\,k}({\mathbb F}) \to {
\mathcal M}_{n\,k}({\mathbb F})$ we denote a linear map defined by
\begin{multline*}
{\mathbf S}_{i\,j}
(\begin{pmatrix}
x_{1\,1} & \cdots & x_{1\,k}
\\
\vdots & \ddots & \vdots
\\
x_{n\,1} & \cdots & x_{n\,k}
\end{pmatrix})
\\
=(-1)^{n-i}\cdot
\begin{pmatrix}
(-1)^{1 - \delta _{1\,j}}x_{i\,j} & \cdots & x_{i\,1} & \cdots & x_{i
\,k}
\\
\vdots & \ddots & \vdots & \ddots & \vdots
\\
(-1)^{1 - \delta _{1\,j}}x_{n\,j} & \cdots & x_{n\,1} & \cdots & x_{n
\,k}
\\
-(-1)^{1 - \delta _{1\,j}}x_{1\,j} & \cdots & -x_{1\,1} & \cdots & -x_{1
\,k}
\\
\vdots & \ddots & \vdots & \ddots & \vdots
\\
-(-1)^{1 - \delta _{1\,j}}x_{(i-1)\,j} & \cdots & -x_{(i-1)\,1} &
\cdots & -x_{(i-1)\,k}
\end{pmatrix}
\end{multline*}
for all $X = (x_{i\,j}) \in {\mathcal M}_{n\,k}({\mathbb F})$. That is,
${\mathbf S}_{i\,j}(X)$ is obtained from $X$ by performing the following
sequence of operations:
\begin{enumerate}
\item
\label{lst:SIJSeq:it1}%
the row cyclical shift sending $i$-th row of $X$ to the first row of the
result;
\item multiplying the bottom $i-1$ rows by $-1$;
\item exchanging the first column and the $j$-th column;
\item
\label{lst:SIJSeq:it4}%
multiplying the first column by $(-1)^{1 - \delta _{1\,j}}$;
\item multiplying all the entries by $(-1)^{n-i}$.
\end{enumerate}
\end{definition}

\begin{lemma}%
\label{lem:SCShiftCommutesWithJoin}%
Let $k, m, n \ge 1$, $A \in {\mathcal M}_{n\,k}({\mathbb F})$,
$B \in {\mathcal M}_{n\,m}({\mathbb F})$ and $1 \le i \le n$. Then
${\mathbf S}_{i\,1}(A|B) = {\mathbf S}_{i\,1}(A)|{\mathbf S}_{i\,1}(B)$.
\end{lemma}
\begin{proof}
Indeed, if
$A \in {\mathcal M}_{n\,k}({\mathbb F}), B \in {\mathcal M}_{n\,m}({
\mathbb F})$ and $1 \le i \le n$, then
\begin{multline*}
{\mathbf S}_{i\,1}(A | B) = {\mathbf S}_{i\,1}
(\begin{pmatrix}
a_{1\,1} & \cdots & a_{1\,k} & b_{1\,1} & \cdots & b_{1\,m}
\\
\vdots & \ddots & \vdots & \vdots & \ddots & \vdots
\\
a_{n\,1} & \cdots & a_{n\,k} & b_{n\,1} & \cdots & b_{n\,m}
\end{pmatrix})
\\
=(-1)^{n-i}\cdot
\begin{pmatrix}
a_{i\,1} & \cdots & a_{i\,k} & b_{i\,1} & \cdots & b_{i\,m}
\\
\vdots & \ddots & \vdots & \vdots & \ddots & \vdots
\\
a_{n\,1} & \cdots & a_{n\,k} & b_{n\,1} & \cdots & b_{n\,m}
\\
-a_{1\,1} & \cdots & -a_{1\,k} & -b_{1\,1} & \cdots & -b_{1\,m}
\\
\vdots & \ddots & \vdots & \vdots & \ddots & \vdots
\\
-a_{(i-1)\,1} & \cdots & -a_{(i-1)\,k} & -b_{(i-1)\,1} & \cdots & -b_{(i-1)
\,m}
\end{pmatrix}
\\
= {\mathbf S}_{i\,1}(A)|{\mathbf S}_{i\,1}(B).\qedhere\quad\ \
\end{multline*}
\end{proof}

\begin{lemma}%
\label{lem:ReduceTo11ByShift}%
Assume that $n \ge k$ and $n + k$ is even. Then ${\mathbf S}_{i\,j}$ is
an invertible linear map preserving $\det _{n\,k}$ for all
$1 \le i \le n$, $1 \le j \le k$.
\end{lemma}

\begin{proof}%
Invertibility of ${\mathbf S}_{i\,j}$ follows directly from the definition.
Assume that $X = (x_{i\,j}) \in {\mathcal M}_{n\,k}({\mathbb F})$. Let
$X' \in {\mathcal M}_{n\,k}({\mathbb F})$ be defined by
\begin{equation*}
X' =
\begin{pmatrix}
(-1)^{1 - \delta _{1\,j}}x_{i\,j} & \cdots & x_{i\,1} & \cdots & x_{i
\,k}
\\
\vdots & \ddots & \vdots & \ddots & \vdots
\\
(-1)^{1 - \delta _{1\,j}}x_{n\,j} & \cdots & x_{n\,1} & \cdots & x_{n
\,k}
\\
-(-1)^{1 - \delta _{1\,j}}x_{1\,j} & \cdots & -x_{1\,1} & \cdots & -x_{1
\,k}
\\
\vdots & \ddots & \vdots & \ddots & \vdots
\\
-(-1)^{1 - \delta _{1\,j}}x_{(i-1)\,j} & \cdots & -x_{(i-1)\,1} &
\cdots & -x_{(i-1)\,k}
\end{pmatrix}
.
\end{equation*}
That is, $X'$ is obtained from $X$ by performing operations \ref{lst:SIJSeq:it1}--\ref{lst:SIJSeq:it4} from the sequence in the definition
of ${\mathbf S}_{i\,j}$ (Definition~\ref{def:SCSDef}). The definition of
$X'$ implies that $X' = (-1)^{n-i}{\mathbf S}_{i\,j}(X)$.

Let $C_{S} \in {\mathcal M}_{k\,k}({\mathbb F})$ be a matrix such that
the linear map $Y \mapsto YC_{S}$ on
${\mathcal M}_{n\,k}({\mathbb F})$ exchanges the first and the $j$-th column
of the matrix and multiplies its first column by
$(-1)^{1 - \delta _{1\,j}}$. That is, $C_{S}$ is the permutation matrix
corresponding to a transposition/trivial permutation $(1j)$ whose first
column is multiplied by $(-1)^{1 - \delta _{1\,j}}$. This implies in particular
that $\det _{k}(C_{S}) = 1$.

On the one hand, we have the following sequence of equalities
%
\begin{equation}
\label{lem:ReduceTo11ByShift:eq1}
\det _{n\,k}({\mathbf S}_{i\,j}(X)) = \det _{n\,k}((-1)^{n-i}X') = (-1)^{(n-i)k}
\det _{n\,k}(X'),
\end{equation}
where the last equality follows from the multilinearity of
$\det _{n\,k}$ with respect to the columns of $X'$.

On the other hand,
\begin{equation*}
XC_{S} =
\begin{pmatrix}
(-1)^{1 - \delta _{1\,j}}x_{1\, j} & \ldots & x_{1\,1} & \ldots & x_{1
\, k}
\\
\vdots & \ddots & \vdots & \ddots & \vdots
\\
(-1)^{1 - \delta _{1\,j}}x_{n\, j} & \ldots & x_{n\,1} & \ldots & x_{n
\, k}
\end{pmatrix}
\end{equation*}
by the definition of $C_{S}$. Hence,
%
\begin{multline}
\label{lem:ReduceTo11ByShiftOdd:eq2}
(-1)^{(n-i)k}\det _{n\,k}(X')
\\
= (-1)^{(n-i)k}
\begin{vmatrix}
(-1)^{1 - \delta _{1\,j}}x_{i\,j} & \cdots & x_{i\,1} & \cdots & x_{i
\,k}
\\
\vdots & \ddots & \vdots & \ddots & \vdots
\\
(-1)^{1 - \delta _{1\,j}}x_{n\,j} & \cdots & x_{n\,1} & \cdots & x_{n
\,k}
\\
-(-1)^{1 - \delta _{1\,j}}x_{1\,j} & \cdots & -x_{1\,1} & \cdots & -x_{1
\,k}
\\
\vdots & \ddots & \vdots & \ddots & \vdots
\\
-(-1)^{1 - \delta _{1\,j}}x_{(i-1)\,j} & \cdots & -x_{(i-1)\,1} &
\cdots & -x_{(i-1)\,k}
\end{vmatrix}_{n\,k}
\\
= \det _{n\,k}(XC_{S}),
\end{multline}
where the last equality follows from Lemma~\ref{lem:PermKNEvenSemicyclic}.
In addition, since $\det _{k}(C_{S}) = 1$, then
%
\begin{equation}
\label{lem:ReduceTo11ByShift:eq2}
\det _{n\,k}(XC_{S}) = \det _{n\,k}(X)
\end{equation}
by Corollary~\ref{cor:rightmatrixmult}.

Therefore,
\begin{equation*}
\det _{n\,k}({\mathbf S}_{i\,j}(X)) = (-1)^{(n-i)k}\det _{n\,k}(X') =
\det _{n\,k}(X)
\end{equation*}
by \eqref{lem:ReduceTo11ByShift:eq1} and \eqref{lem:ReduceTo11ByShift:eq2}. From this we conclude that
${\mathbf S}_{i\,j}$ preserves $\det _{n\,k}$.
\end{proof}

\begin{lemma}%
\label{lem:NKEvenAdditionPreserverIsZero}%
Assume that $n \ge k \ge 2$, $|{\mathbb F}| > k$ and $n + k$ is even. Then
$\rad (\det _{n\,k}) = \{0\}$.
\end{lemma}
\begin{proof}%
Assume that
$Y  \in \rad (\det _{n\,k}) \subseteq {\mathcal M}_{n\, k}({\mathbb F})$. We will
show that $y_{i\,j} = 0$ for all $1 \le i \le n, 1 \le j \le k$.

Now consider ${\mathbf S}_{i\,j}(Y)$ where ${\mathbf S}_{i\,j}$ is an invertible
linear map defined in Definition~\ref{def:SCSDef}. The definition of
${\mathbf S}_{i\,j}$ implies that
%
\begin{equation}
\label{lem:NKEvenAdditionPreserverIsZero:eq1}
{\mathbf S}_{i\,j}(Y)_{1\,1} = (-1)^{n-i+1-\delta _{1\,j}}y_{i\,j}.
\end{equation}

In addition,
\begin{multline*}
\det _{n\,k} \left (A + \lambda {\mathbf S}_{i\,j}(Y)\right ) = \det _{n
\,k} \left ({\mathbf S}_{i\,j}({\mathbf S}_{i\,j}^{-1}(A)) + \lambda {
\mathbf S}_{i\,j}(Y)\right )
= \det _{n\,k} \left ({\mathbf S}_{i\,j}({\mathbf S}_{i\,j}^{-1}(A) +
\lambda Y)\right )\\
= \det _{n\,k} \left ({\mathbf S}_{i\,j}^{-1}(A) +
\lambda Y\right )
= \det _{n\,k} \left ({\mathbf S}_{i\,j}^{-1}(A)\right ) = \det _{n\,k}(A)\;\;\mbox{for all}\;\;A \in {\mathcal M}_{n\,k}({\mathbb F}),\,
\lambda \in {\mathbb F}.
\end{multline*}
Hence, ${\mathbf S}_{i\,j}(Y)$ satisfies the
condition of Lemma~\ref{lem:NKEvenAdditionPreserverHasZero11}. This implies
that
%
\begin{equation}
\label{lem:NKEvenAdditionPreserverIsZero:eq2}
{\mathbf S}_{i\,j}(Y)_{1\,1} = 0.
\end{equation}
By aligning together the equalities~\eqref{lem:NKEvenAdditionPreserverIsZero:eq1}
and \eqref{lem:NKEvenAdditionPreserverIsZero:eq2} we obtain that
\begin{equation*}
(-1)^{n-i+1-\delta _{1\,j}}y_{i\,j} = {\mathbf S}_{i\,j}(Y)_{1\,1} = 0.
\end{equation*}

Thus, $y_{i\,j} = 0$ for all $1 \le i \le n$ and $1 \le j \le k$. Therefore,
$Y = 0$ and consequently $\rad (\det _{n\,k}) = \{0\}$.
\end{proof}

From the lemma above we conclude that we could apply Lemma~\ref{lem:RadicalCriterion}
and deduce the corollary below.

\begin{corollary}%
\label{cor:LinPresIso}%
Assume that $k \ge 2$, $n \ge k + 2$, $|{\mathbb F}| > k$ and $n + k$ is even. Let
$T$ be a linear map on ${\mathcal M}_{n\,k}({\mathbb F})$ such that
$\det _{n\,k}(T(X)) = \det _{n\,k}(X)$ for all
$X \in {\mathcal M}_{n\,k}({\mathbb F})$. Then $T$ is invertible.
\end{corollary}

\begin{note}%
\label{note4.8}
The statement of Corollary~\ref{cor:LinPresIso} does not hold for
$n + k$ odd because in this case $\rad (\det _{n\,k}) \neq \{O_{n\,k}\}$. Indeed,
Lemma~\ref{lem:NAKAGAMILEMMA20} implies that
\begin{equation*}
O_{n\,k} \neq
\begin{pmatrix}
1 & \cdots & 1
\\
\vdots & \ddots & \vdots
\\
1 & \cdots & 1
\end{pmatrix}
\in \rad (\det _{n\,k})
\end{equation*}
if $n - (k + 1)$ is even.

For this reason we treat this case separately in the subsequent paper.
\end{note}

\section{Proof of the main theorem}
\label{sec:maintheorem}

We begin with establishing a relationship between the following two properties
of a given matrix $B \in {\mathcal M}_{n\,k}({\mathbb F})$: ``$\deg _{
\lambda} (\det _{n\, k} (A + \lambda B)) \le 1$ for all
$A \in {\mathcal M}_{n\,k}({\mathbb F})$'' and ``$\rk (B) \le 1$''. Since
the first property is preserved by every linear $\det _{n\,k}$-preserver
as it is shown in Lemma~\ref{lem:DetNKPresPresLowDeg} below, then we could
use this relationship to prove that the second property is also preserved
by every linear $\det _{n\,k}$-preserver. This would allow us to use the
results of Westwick regarding linear maps preserving the set of matrices
of rank one.

\begin{lemma}%
\label{lem:Rank1CullisDeg1}%
Let $n \ge k$ and $B \in {\mathcal M}_{n\,k}({\mathbb F})$ be such that
$\rk (B) \le 1$. Then
\begin{equation*}
\deg _{\lambda} (\det _{n\, k} (A + \lambda B)) \le 1\;\;
\text{for all}\;\;A \in {\mathcal M}_{n\, k}({\mathbb F}).
\end{equation*}
\end{lemma}
\begin{proof}
Let $P \in {\mathbb F}[\lambda ]$ be a polynomial defined by
\begin{equation*}
P = \det _{n\, k} (A + \lambda B) = a_{0} + a_{1}\lambda + \ldots + a_{k}
\lambda ^{k},
\end{equation*}
where $a_{0}, \ldots , a_{k} \in {\mathbb F}$. Let us show that
$a_{d} = 0$ for all $d \ge 2$. Indeed, the expansion~\eqref{eq:CullisBinomialExpansion}
implies that
%
\begin{equation}
\label{lem:Rank1CullisDeg1:eq1}
a_{d} = \sum _{1\le i_{1} < \ldots < i_{d} \le k} \det _{n\,k}\Bigl(A(|1]
\Big|\ldots \Big| B(|i_{1}] \Big| \ldots \Big| B(|i_{d}] \Big|
\ldots \Big| A(|k] \Bigr).
\end{equation}
Assume that $d \ge 2$. Let $1\le i_{1} < \ldots < i_{d} \le k$ and
$C \in {\mathcal M}_{n\,k}({\mathbb F})$ be defined by
\begin{equation*}
C = A(|1]\Big|\ldots \Big| B(|i_{1}] \Big| \ldots \Big| B(|i_{d}]
\Big| \ldots \Big| A(|k].
\end{equation*}
This matrix contains at least two columns from $B$ because $d \ge 2$. Since
$\rk (B) \le 1$, these columns of $C$ are linearly dependent and consequently
$\det _{n\,k}(C) = 0$. This implies that the right-hand side of~\eqref{lem:Rank1CullisDeg1:eq1}
is equal to 0. Hence, $a_{d} = 0$ as the left-hand side of~\eqref{lem:Rank1CullisDeg1:eq1}.

Therefore, $a_{d} = 0$ for all $d \ge 2$ and
$\deg _{\lambda} (\det _{n\, k} (A + \lambda B)) \le 1$.
\end{proof}

\begin{remark}
\label{rem5.2}
The converse of Lemma~\ref{lem:Rank1CullisDeg1} is also true under additional
assumptions. It is proved in Lemma~\ref{lem:CullisDeg1Rank1}, but the proof
is more complicated. It is easy to show using the expansion~\eqref{eq:CullisBinomialExpansion}
that for a given $B \in {\mathcal M}_{n\,k}({\mathbb F})$ the inequality
\begin{equation*}
\deg _{\lambda} (\det _{n\, k} (A + \lambda B)) \le 1\;\;
\text{for all}\;\; A \in {\mathcal M}_{n\,k}({\mathbb F})
\end{equation*}
implies that
\begin{equation*}
\det _{n\,k}\Bigl(B(|i]\Big|B(|j]\Big| A'\Bigr) = 0\;\;\text{for all}
\;\; 1 \le i < j \le k,\, A' \in {\mathcal M}_{n\,(k-2)}({\mathbb F}).
\end{equation*}

One would try to deduce from this that
%
\begin{equation}
\label{eq:Conv1}
\det _{2} \Bigl(B[l,m|i,j]\Bigr) = 0\;\;\;\;\text{for all}\;\;1 \le i <
j \le k\;\;\text{and}\;\;1 \le l < m \le n
\end{equation}
by finding the appropriate $A'$ such that
\begin{equation*}
\det _{n\,k}\Bigl(B(|i]\Big|B(|j]\Big| A'\Bigr) = \det _{2} \Bigl(B[l,
m|i, j]\Bigr).
\end{equation*}
Then the condition~\eqref{eq:Conv1} would easily imply that
$\rk (B) \le 1$.

Unfortunately, our computations show that if $k = 4$ and
$\Char ({\mathbb F}) \neq 0$, then for certain $n \ge k + 2$ such
$A'$ does not exist.

Nevertheless, the set of possible sums of $2\times 2$ minors of $X$ obtained
by substituting different $A' \in {\mathcal M}_{n\,2}({\mathbb F})$ in
$\det _{n\, 4}(X|A')$ is large enough for our purposes. It is possible
to obtain the equality~\eqref{eq:Conv1} from the equalities
$\det _{n\, 4}(X|A') = 0$ for all
$A' \in {\mathcal M}_{n\,2}({\mathbb F})$ belonging to this set as it is
done below. The simplest expression of the form
$\det _{n\, 4}(X|A')$ could be obtained from the next lemma.
\end{remark}

\begin{lemma}%
\label{lem:Exists2DiffDiff}
Let $n \ge k \ge 4$, $n > l > 2$ and
$X = (x_{i\,j}) \in {\mathcal M}_{n\, 2}({\mathbb F})$. Then there exists
$B \in {\mathcal M}_{n\,(k-2)}({\mathbb F})$ such that
\begin{equation*}
\det _{n\, k} (X | B) = \left |
\begin{matrix}
x_{1\, 1} - x_{2\, 1} & x_{1\, 2} - x_{2\, 2}
\\
x_{l\, 1} - x_{(l+1)\, 1} & x_{l\, 2} - x_{(l+1)\, 2}
\end{matrix}
\right |.
\end{equation*}
\end{lemma}
\begin{proof}
Since
$\left |\{1, \ldots , n\} \setminus \{1, 2, l, l+1\}\right | = n - 4
\ge k - 4$, then there exists a sequence
$2 < i_{1} < \ldots < i_{k-4} \le n$ such that
$\{l, l+1\} \cap \{i_{1},\ldots , i_{k-4}\} = \varnothing $. Let
$A \in {\mathcal M}_{n\, (k-2)}({\mathbb F})$ be defined by
\begin{equation*}
A = E_{1\, 1} + E_{2\, 1} + E_{l\, 2} + E_{(l+1)\, 2} + \sum _{
\alpha = 1}^{k-4} E_{i_{\alpha}\, (\alpha + 2)},
\end{equation*}
where we assume that the last summand is equal to zero when $k = 4$.

Using the Laplace expansion of $\det _{n\, k} (X | A)$ along the last column
$k-4$ times we obtain that
%
\begin{equation}
\label{eq:DetYDetY}
\det _{n\, k} (X | A) = (-1)^{s} \det _{(n-k+4)\,4} (Y)
\end{equation}
for some $s \in \mathbb N$ and
$Y\in {\mathcal M}_{(n-k+4)\,4}({\mathbb F})$ defined by
\begin{equation*}
Y = \Bigl(X(|1] \Big| X(|2] \Big| A(|1] \Big| A(|2]\Bigr)(i_{1},
\ldots , i_{k-4}|).
\end{equation*}

Note that $Y$ has the form
\begin{equation*}
\begin{pmatrix}
x_{1\, 1} & x_{1\, 2} & 1 & 0
\\
x_{2\, 1} & x_{2\, 2} & 1 & 0
\\
\vdots & \vdots & \vdots & \vdots
\\
x_{l\, 1}& x_{l\, 2} & 0 & 1
\\
x_{(l+1)\, 1}& x_{(l+1)\,2} & 0 & 1
\\
\vdots & \vdots & \vdots & \vdots
\end{pmatrix}
.
\end{equation*}
This implies that $Y$ contains the pairs of consecutive rows
\begin{equation*}
(x_{1\,1}, x_{1\,2}, 1, 0),\; (x_{2\,1}, x_{2\,2}, 1, 0)\quad
\text{and} \quad (x_{l\,1}, x_{l\,2}, 0, 1),\; (x_{(l+1)\,1}, x_{(l+1)
\,2}, 0, 1).
\end{equation*}
Suppose that the last pair is located in $Y$ on the $l'$-th and the
$(l'+1)$-th row. Then it follows from the Laplace expansion along the last
column of $Y$ that
\begin{equation*}
\det _{(n-k+4)\,4}(Y) = (-1)^{l'}
\begin{vmatrix}
x_{1\, 1} & x_{1\, 2} & 1
\\
x_{2\, 1} & x_{2\, 2} & 1
\\
\vdots & \vdots & \vdots
\\
x_{l\, 1}& x_{l\,2} & 0
\\[-1.5ex]
\hline
\\[-2.5ex]
\\
x_{(l+1)\, 1}& x_{(l+1)\, 2} & 0
\\
\vdots & \vdots & \vdots
\end{vmatrix}_{(n-k+3)\, 3}
\hspace{-28pt} - (-1)^{l'}
\begin{vmatrix}
x_{1\, 1} & x_{1\, 2} & 1
\\
x_{2\, 1} & x_{2\, 2} & 1
\\
\vdots & \vdots & \vdots
\\
x_{l\, 1}& x_{l\,2} & 0
\\
x_{(l+1)\, 1}& x_{(l+1)\,2} & 0
\\[-1.5ex]
\hline
\\[-2.5ex]
\\
\vdots & \vdots & \vdots
\end{vmatrix}_{(n-k+3)\, 3}.
\end{equation*}

Next, if we apply the Laplace expansion along the last column to each summand
separately, then we obtain that
\begin{equation*}
\begin{vmatrix}
x_{1\, 1} & x_{1\, 2} & 1
\\
x_{2\, 1} & x_{2\, 2} & 1
\\
\vdots & \vdots & \vdots
\\
x_{l\, 1}& x_{l\,2} & 0
\\[-1.5ex]
\hline
\\[-2.5ex]
\\
x_{(l+1)\, 1}& x_{(l+1)\, 2} & 0
\\
\vdots & \vdots & \vdots
\end{vmatrix}_{(n-k+3)\, 3}
\hspace{-8pt} =
\begin{vmatrix}
x_{1\, 1} & x_{1\, 2}
\\[-1.5ex]
\hline
\\[-2.5ex]
\\
x_{2\, 1} & x_{2\, 2}
\\
\vdots & \vdots
\\
x_{l\, 1}& x_{l\,2}
\\[-1.5ex]
\hline
\\[-2.5ex]
\\
x_{(l+1)\, 1}& x_{(l+1)\,2}
\\
\vdots & \vdots
\end{vmatrix}_{(n-k+2)\, 2} \hspace{-8pt} -
\begin{vmatrix}
x_{1\, 1} & x_{1\, 2}
\\
x_{2\, 1} & x_{2\, 2}
\\[-1.5ex]
\hline
\\[-2.5ex]
\\
\vdots & \vdots
\\
x_{l\, 1}& x_{l\,2}
\\[-1.5ex]
\hline
\\[-2.5ex]
\\
x_{(l+1)\, 1}& x_{(l+1)\,2}
\\
\vdots & \vdots
\end{vmatrix}_{(n-k+2)\, 2}
\end{equation*}
and
\begin{equation*}
\begin{vmatrix}
x_{1\, 1} & x_{1\, 2} & 1
\\
x_{2\, 1} & x_{2\, 2} & 1
\\
\vdots & \vdots & \vdots
\\
x_{l\, 1}& x_{l\,2} & 0
\\
x_{(l+1)\, 1}& x_{(l+1)\, 2} & 0
\\[-1.5ex]
\hline
\\[-2.5ex]
\\
\vdots & \vdots & \vdots
\end{vmatrix}_{(n-k+3)\, 3}
\hspace{-8pt}
=
\begin{vmatrix}
x_{1\, 1} & x_{1\, 2}
\\[-1.5ex]
\hline
\\[-2.5ex]
\\
x_{2\, 1} & x_{2\, 2}
\\
\vdots & \vdots
\\
x_{l\, 1}& x_{l\,2}
\\
x_{(l+1)\, 1}& x_{(l+1)\, 2}
\\[-1.5ex]
\hline
\\[-2.5ex]
\\
\vdots & \vdots
\end{vmatrix}_{(n-k+2)\, 2} \hspace{-8pt} -
\begin{vmatrix}
x_{1\, 1} & x_{1\, 2}
\\
x_{2\, 1} & x_{2\, 2}
\\[-1.5ex]
\hline
\\[-2.5ex]
\\
\vdots & \vdots
\\
x_{l\, 1}& x_{l\,2}
\\
x_{(l+1)\, 1}& x_{(l+1)\, 2}
\\[-1.5ex]
\hline
\\[-2.5ex]
\\
\vdots & \vdots
\end{vmatrix}_{(n-k+2)\, 2}.
\end{equation*}

Thus,
\begin{multline*}
(-1)^{l'}\det _{(n-k+4)\, 4}(Y)
\\
=
\begin{vmatrix}
x_{1\, 1} & x_{1\, 2}
\\[-1.5ex]
\hline
\\[-2.5ex]
\\
x_{2\, 1} & x_{2\, 2}
\\
\vdots & \vdots
\\
x_{l\, 1}& x_{l\,2}
\\[-1.5ex]
\hline
\\[-2.5ex]
\\
x_{(l+1)\, 1}& x_{(l+1)\,2}
\\
\vdots & \vdots
\end{vmatrix}_{(n-k+2)\, 2} -
\begin{vmatrix}
x_{1\, 1} & x_{1\, 2}
\\
x_{2\, 1} & x_{2\, 2}
\\[-1.5ex]
\hline
\\[-2.5ex]
\\
\vdots & \vdots
\\
x_{l\, 1}& x_{l\,2}
\\[-1.5ex]
\hline
\\[-2.5ex]
\\
x_{(l+1)\, 1}& x_{(l+1)\,2}
\\
\vdots & \vdots
\end{vmatrix}_{(n-k+2)\, 2}
\\
-
\begin{vmatrix}
x_{1\, 1} & x_{1\, 2}
\\[-1.5ex]
\hline
\\[-2.5ex]
\\
x_{2\, 1} & x_{2\, 2}
\\
\vdots & \vdots
\\
x_{l\, 1}& x_{l\,2}
\\
x_{(l+1)\, 1}& x_{(l+1)\, 2}
\\[-1.5ex]
\hline
\\[-2.5ex]
\\
\vdots & \vdots
\end{vmatrix}_{(n-k+2)\, 2} +
\begin{vmatrix}
x_{1\, 1} & x_{1\, 2}
\\
x_{2\, 1} & x_{2\, 2}
\\[-1.5ex]
\hline
\\[-2.5ex]
\\
\vdots & \vdots
\\
x_{l\, 1}& x_{l\,2}
\\
x_{(l+1)\, 1}& x_{(l+1)\, 2}
\\[-1.5ex]
\hline
\\[-2.5ex]
\\
\vdots & \vdots
\end{vmatrix}_{(n-k+2)\, 2}
\\
= \det _{(n-k+2)\,2}(K) - \det _{(n-k+2)\,2}(L) - \det _{(n-k+2)\,2}(M)
+ \det _{(n-k+2)\,2}(N),
\end{multline*}
where
\begin{equation*}
K = Y(1,l|3,4),\quad L = Y(2,l|3,4),\quad M = Y(1,l+1|3,4),\quad N = Y(2,l+1|3,4).
\end{equation*}
Let us express each Cullis' determinant as a sum of minors applying Corollary~\ref{lem:DetNKAsSumProj}.
\begin{multline*}
\det _{(n-k+2)\,2}(K) - \det _{(n-k+2)\,2}(L) - \det _{(n-k+2)\,2}(M) +
\det _{(n-k+2)\,2}(N)
\\
= \hspace{-10pt} \sum _{c \in \binom{[n-k+2]}{2}}\hspace{-10pt}\sgn _{[n-k+2]}(c) \biggl(\det _{2}
\Bigl(K[c|)\Bigr) - \det _{2}\Bigl(L[c|)\Bigr)
- \det _{2}\Bigl(M[c|)\Bigr) + \det _{2}\Bigl(N[c|)\Bigr)\biggr).
\end{multline*}

Note that if $c \neq \{1, l' - 1\}$, then
\begin{equation*}
\det _{2}\Bigl(K[c|)\Bigr) - \det _{2}\Bigl(L[c|)\Bigr) - \det _{2}
\Bigl(M[c|)\Bigr) + \det _{2}\Bigl(N[c|)\Bigr) = 0
\end{equation*}
because if $c \cap \{1, l' - 1\} = \varnothing $, then
\begin{equation*}
\det _{2}\Bigl(K[c|)\Bigr) = \det _{2}\Bigl(L[c|)\Bigr) = \det _{2}
\Bigl(M[c|)\Bigr) = \det _{2}\Bigl(N[c|)\Bigr),
\end{equation*}
and if $c \cap \{1, l' - 1\} = \{1\}$ or $\{l'-1\}$, then
\begin{equation*}
\det _{2}\Bigl(K[c|)\Bigr) = \det _{2}\Bigl(M[c|)\Bigr)\quad
\text{and}\quad \det _{2}\Bigl(L[c|)\Bigr) = \det _{2}\Bigl(N[c|)
\Bigr)
\end{equation*}
or
\begin{equation*}
\det _{2}\Bigl(K[c|)\Bigr) = \det _{2}\Bigl(L[c|)\Bigr) \quad
\text{and}\quad \det _{2}\Bigl(M[c|)\Bigr) = \det _{2}\Bigl(N[c|)
\Bigr),
\end{equation*}
correspondingly. It follows
%
\begin{multline}
\label{lem:Exists2DiffDiff:eqq1}
(-1)^{l'}\det _{(n-k+4)\, 4}(Y) = (-1)^{1 + (l'-1) - 1 - 2}\biggl(
\det _{2}\Bigl(K[1,l'-1|)\Bigr)
\\
- \det _{2}\Bigl(L[1,l'-1|)\Bigr) - \det _{2}\Bigl(M[1,l'-1|)\Bigr) +
\det _{2}\Bigl(N[1,l'-1|)\Bigr)\biggr).
\end{multline}

Let us simplify the right-hand side of~\eqref{lem:Exists2DiffDiff:eqq1}.
\begin{enumerate}[label=\textbf{\roman{enumi}.}, ref=\textbf{\roman{enumi}.}]
\item Express all the $2\times 2$-determinants in terms of entries of
$X$
%
\begin{multline}
\label{lem:Exists2DiffDiff:eqq2}
\det _{2}\Bigl(K[1, l'-1|)\Bigr) - \det _{2}\Bigl(L[1, l'-1|)\Bigr)
- \det _{2}\Bigl(M[1, l'-1|)\Bigr) + \det _{2}\Bigl(N[1, l'-1|)\Bigr)
\\
=
\begin{vmatrix}
x_{2\, 1} & x_{2\, 2}
\\
x_{(l+1)\, 1} & x_{(l+1)\, 2}
\end{vmatrix} -
\begin{vmatrix}
x_{1\, 1} & x_{1\, 2}
\\
x_{(l+1)\, 1} & x_{(l+1)\, 2}
\end{vmatrix} -
\begin{vmatrix}
x_{2\, 1} & x_{2\, 2}
\\
x_{l\, 1} & x_{l\, 2}
\end{vmatrix} +
\begin{vmatrix}
x_{1\, 1} & x_{1\, 2}
\\
x_{l\, 1} & x_{l\, 2}
\end{vmatrix}.
\end{multline}
\item Using the multilinearity of $2\times 2$-determinant along the first
row, we obtain
%
\begin{multline}
\label{lem:Exists2DiffDiff:eqq3}
\begin{vmatrix}
x_{2\, 1} & x_{2\, 2}
\\
x_{(l+1)\, 1} & x_{(l+1)\, 2}
\end{vmatrix} -
\begin{vmatrix}
x_{1\, 1} & x_{1\, 2}
\\
x_{(l+1)\, 1} & x_{(l+1)\, 2}
\end{vmatrix} -
\begin{vmatrix}
x_{2\, 1} & x_{2\, 2}
\\
x_{l\, 1} & x_{l\, 2}
\end{vmatrix} +
\begin{vmatrix}
x_{1\, 1} & x_{1\, 2}
\\
x_{l\, 1} & x_{l\, 2}
\end{vmatrix}
\\
=
\begin{vmatrix}
x_{2\, 1} - x_{1\,1}& x_{2\, 2} - x_{1\,2}
\\
x_{(l+1)\, 1} & x_{(l+1)\, 2}
\end{vmatrix} -
\begin{vmatrix}
x_{2\, 1} - x_{1\,1}& x_{2\, 2} - x_{1\,2}
\\
x_{l\, 1} & x_{l\, 2}
\end{vmatrix}.
\end{multline}
\item Using the multilinearity of $2\times 2$-determinant along the second
row, we have
%
\begin{multline}
\label{lem:Exists2DiffDiff:eqq4}
\begin{vmatrix}
x_{2\, 1} - x_{1\,1}& x_{2\, 2} - x_{1\,2}
\\
x_{(l+1)\, 1} & x_{(l+1)\, 2}
\end{vmatrix} -
\begin{vmatrix}
x_{2\, 1} - x_{1\,1}& x_{2\, 2} - x_{1\,2}
\\
x_{l\, 1} & x_{l\, 2}
\end{vmatrix}
\\
=
\begin{vmatrix}
x_{2\, 1} - x_{1\,1}& x_{2\, 2} - x_{1\,2}
\\
x_{(l+1)\, 1} - x_{l\, 1}& x_{(l+1)\, 2} - x_{l\, 2}
\end{vmatrix} =
\begin{vmatrix}
x_{1\, 1} - x_{2\,1}& x_{1\, 2} - x_{2\,2}
\\
x_{l\, 1} - x_{(l+1)\, 1}& x_{l\, 2} - x_{(l+1)\, 2}
\end{vmatrix}.
\end{multline}
\end{enumerate}
Thus,
\begin{equation*}
\det _{(n-k+4)\, 4}(Y) = -
\begin{vmatrix}
x_{1\, 1} - x_{2\,1}& x_{1\, 2} - x_{2\,2}
\\
x_{l\, 1} - x_{(l+1)\, 1}& x_{l\, 2} - x_{(l+1)\, 2}
\end{vmatrix}
\end{equation*}
by~\eqref{lem:Exists2DiffDiff:eqq1} and~\eqref{lem:Exists2DiffDiff:eqq2}--\eqref{lem:Exists2DiffDiff:eqq4}.
This implies that if $B$ is obtained from $A$ by multiplying its last column
by $(-1)^{(s+1)}$, then
\begin{multline*}
\det _{n\, k} (X|B) = (-1)^{(s+1)}\det _{n\, k} (X|A)
\\
= -\det _{(n-k+4)\, 4}(Y) =
\begin{vmatrix}
x_{1\, 1} - x_{2\, 1}& x_{1\, 2} - x_{2\, 2}
\\
x_{l\, 1} - x_{(l+1)\, 1}& x_{l\, 2} - x_{(l+1)\,2}
\end{vmatrix}.\qedhere\quad\ \
\end{multline*}
\end{proof}

\begin{remark}
\label{rem5.4}
If  $X \in {\mathcal M}_{n\,2}({\mathbb F})$ and we choose $B  \in {\mathcal M}_{n\,(k-2)}({\mathbb F})$ as it is
done in the above lemma, then the condition  $\det _{n\,k}(X|B) = 0$ for all such 
$B$ does
not imply that $\rk \Bigl(X[1,2|)\Bigr) \le 1$. 

Indeed, if
$X =
\begin{psmallmatrix}
1 & 0
\\
2 & 1
\\
3 & 2
\\
\vdots & \vdots
\\
n & (n-1)
\end{psmallmatrix}
\in {\mathcal M}_{n\,2}({\mathbb F})$, then all the $2\times 2$-determinants
formed by the differences of consecutive rows of $X$ are zero but
$\rk \Bigl(X[1,2|)\Bigr) = 2$ in every ground field. The next two lemmas
yield the additional matrices in
${\mathcal M}_{n\,(k-2)}({\mathbb F})$ which are used to conclude that
$\rk \Bigl(X[1,2|)\Bigr) \le 1$.
\end{remark}
%
\begin{lemma}%
\label{lem:Exists2DiffSum}
Assume that $n\ge k \ge 3$. Let
$X \in {\mathcal M}_{n\, 2}({\mathbb F})$. Then there exists
$B \in {\mathcal M}_{n\, (k-2)}({\mathbb F})$ such that
\begin{equation*}
\det _{n\, k} (X | B) = \sum _{2 < l \le n-k+3} (-1)^{l}
\begin{vmatrix}
x_{1\, 1} - x_{2\, 1} & x_{1\, 2} - x_{2\, 2}
\\
x_{l\, 1} & x_{l\, 2}
\end{vmatrix}
.
\end{equation*}
\end{lemma}
\begin{proof}
Let $A \in {\mathcal M}_{n\, (k-2)}({\mathbb F})$ be defined by
\begin{equation*}
A = E_{1\, 1} + E_{2\, 1} + \left (E_{(n-k+4)\, 2} + \ldots + E_{n\, (k-2)}
\right ) = E_{1\, 1} + E_{2\, 1} + \sum _{i=1}^{k-3}E_{(n-k+3+i)\,i+1},
\end{equation*}
where we assume that the last summand is equal to zero when $k = 3$. It
follows from the Laplace expansion of $\det _{n\, k} (X | A)$ along the
last column $k-3$ times that
%
\begin{equation}
\label{eq:DetXBDetZ}
\det _{n\, k} (X | A) = (-1)^{s} \det _{(n-k+3)\, 3} (Y)
\end{equation}
for some $s \in \mathbb N$ and
$Y =
\begin{psmallmatrix}
x_{1\, 1} & x_{1\, 2} & 1
\\
x_{2\, 1} & x_{2\, 2} & 1
\\
\vdots & \vdots & \vdots
\\
x_{(n-k+3)\, 1} & x_{(n-k+3)\, 2} & 0
\end{psmallmatrix}.$
Then the Laplace expansion along the last column of $Y$ implies that
\begin{multline*}
\det _{(n-k+3)\, 3} (Y)
\\
=
\begin{vmatrix}
x_{2\, 1} & x_{2\, 2}
\\
\vdots & \vdots
\\
x_{(n-k+3)\, 1} & x_{(n-k+3)\, 2}
\end{vmatrix}_{(n-k+2)\, 2} -
\begin{vmatrix}
x_{1\, 1} & x_{1\, 2}
\\
\vdots & \vdots
\\
x_{(n-k+3)\, 1} & x_{(n-k+3)\, 2}
\end{vmatrix}_{(n-k+2)\, 2}
\\
= \det _{(n-k+2)\,2}(Y') - \det _{(n-k+2)\,2}(Y''),
\end{multline*}
where $Y' = Y(1|3),\; Y'' = Y(2|3)$. Let us express each Cullis' determinant
as a sum of minors following Corollary~\ref{lem:DetNKAsSumProj}.
\begin{multline*}
\det _{(n-k+2)\,2}(Y') - \det _{(n-k+2)\,2}(Y'')
\\
= \sum _{c \in \binom{[n-k+2]}{2}}\sgn _{[n - k + 2]}(c) \left (\det _{2}
\Bigl(Y'[c|)\Bigr) - \det _{2}\Bigl(Y''[c|)\Bigr)\right ).
\end{multline*}

Note that if $1 \not \in c$, then $Y'[c|) = Y''[c|)$. Hence,
$\det _{2}\Bigl(Y'[c|)\Bigr) = \det _{2}\Bigl(Y''[c|)\Bigr)$ or, equivalently,
\begin{equation*}
\det _{2}\Bigl(Y'[c|)\Bigr) - \det _{2}\Bigl(Y''[c|)\Bigr) = 0\;\;\mbox{for all}\;\;1 \not\in c \in \binom{[n-k+2]}{2}.
\end{equation*}
Thus,
\begin{multline*}
\det _{(n-k+3)\, 3} (Y) = \det _{(n-k+2)\,2}(Y') - \det _{(n-k+2)\,2}(Y'')
\\
= \sum _{c \in \binom{[n-k+2]}{2}}\sgn _{[n - k + 2]}(c) \left (\det _{2}
\Bigl(Y'[c|)\Bigr) - \det _{2}\Bigl(Y''[c|)\Bigr)\right )
\\
= \sum _{\substack{c \in \binom{[n-k+2]}{2}\\1 \not \in c}}\sgn _{[n -
k + 2]}(c) \left (\det _{2}\Bigl(Y'[c|)\Bigr) - \det _{2}\Bigl(Y''[c|)
\Bigr)\right )
\\
+ \sum _{\substack{c \in \binom{[n-k+2]}{2}\\1 \in c}}\sgn _{[n - k + 2]}(c)
\left (\det _{2}\Bigl(Y'[c|)\Bigr) - \det _{2}\Bigl(Y''[c|)\Bigr)
\right )
\\
= 0 + \sum _{\substack{c \in \binom{[n-k+2]}{2}\\1 \in c}}\sgn _{[n - k
+ 2]}(c) \left (\det _{2}\Bigl(Y'[c|)\Bigr) - \det _{2}\Bigl(Y''[c|)
\Bigr)\right )
\end{multline*}

Using the fact that for $c = \{1, l'\} \in [n - k + 2]$ we have
$Y'[c|) = Y[2,l'+1|3)$ and $Y''[c|) = Y[1,l'+1|3)$, we obtain that
\begin{multline*}
\det _{(n-k+3)\, 3} (Y)
= \sum _{\substack{c \in \binom{[n-k+2]}{2}\\1 \in c}}\sgn _{[n - k + 2]}(c)
\left (\det _{2}\Bigl(Y'[c|)\Bigr) - \det _{2}\Bigl(Y''[c|)\Bigr)
\right )
\\
= \sum _{2 \le l' \le n - k + 2}(-1)^{1 + l' - 1 - 2}\left (
\begin{vmatrix}
x_{2\, 1} & x_{2\, 2}
\\
x_{(l'+1)\, 1} & x_{(l'+1)\, 2}
\end{vmatrix} -
\begin{vmatrix}
x_{1\, 1} & x_{1\, 2}
\\
x_{(l'+1)\, 1} & x_{(l'+1)\, 2}
\end{vmatrix} \right )
\\
\overset{(l'+1) \mapsto l}{\mylongjoin}
\sum _{2 < l \le n - k + 3}(-1)^{1 + (l-1) - 1 - 2}\left (
\begin{vmatrix}
x_{2\, 1} & x_{2\, 2}
\\
x_{l\, 1} & x_{l\, 2}
\end{vmatrix} -
\begin{vmatrix}
x_{1\, 1} & x_{1\, 2}
\\
x_{l\, 1} & x_{l\, 2}
\end{vmatrix} \right )
\end{multline*}

The multilinearity of $2\times 2$-determinant along the first row implies
that
\begin{multline*}
\det _{(n-k+3)\, 3} (Y) = \sum _{2 < l \le n - k + 3}
(-1)^{1 + (l-1)
- 1 - 2}\left (
\begin{vmatrix}
x_{2\, 1} & x_{2\, 2}
\\
x_{l\, 1} & x_{l\, 2}
\end{vmatrix} -
\begin{vmatrix}
x_{1\, 1} & x_{1\, 2}
\\
x_{l\, 1} & x_{l\, 2}
\end{vmatrix} \right )
\\
= \sum _{2 < l \le n - k + 3}(-1)^{(l-1)}
\begin{vmatrix}
x_{2\, 1} - x_{1\, 1} & x_{2\, 2} - x_{1\, 2}
\\
x_{l\, 1} & x_{l\, 2}
\end{vmatrix}.
\end{multline*}

From this we conclude that if $B$ is obtained from $A$ by multiplying its
last column by $(-1)^{s-1}$, then
\begin{multline*}
\det _{n\, k} (X|B) = (-1)^{(s-1)}\det _{n\, k} (X|A) = -\det _{(n-k+3)
\, 3} (Y)
\\
= \sum _{2 < l \le n - k + 3}(-1)^{l}
\begin{vmatrix}
x_{1\, 1} - x_{2\, 1} & x_{1\, 2} - x_{2\, 2}
\\
x_{l\, 1} & x_{l\, 2}
\end{vmatrix}.\qedhere\quad\ \
\end{multline*}
\end{proof}

\begin{lemma}%
\label{lem:Exists2Sum}
Assume that $n \ge k \ge 2$. If
$X \in {\mathcal M}_{n\, 2}({\mathbb F})$, then there exists
$B \in {\mathcal M}_{n\, (k-2)}({\mathbb F})$ such that
%
\begin{multline}
\label{eq:LemmaX2KExpansionStatement}
\det _{n\, k} (X | B) =
\begin{vmatrix}
x_{1\, 1} & x_{1\, 2}
\\
x_{2\, 1} & x_{2\, 2}
\end{vmatrix} + \sum _{l = 3}^{n-k+2} (-1)^{l-2}
\begin{vmatrix}
x_{1\, 1} - x_{2\, 1}& x_{1\, 2} - x_{2\, 2}
\\
x_{l\, 1} & x_{l\, 2}
\end{vmatrix}
\\
+ \sum _{2 < l < m \le n-k+2} (-1)^{l + m - 3}
\begin{vmatrix}
x_{l\, 1} & x_{l\, 2}
\\
x_{m\, 1} & x_{m\, 2}
\end{vmatrix}.
\end{multline}
\end{lemma}
\begin{proof}%
The statement of the lemma clearly holds if $k = 2$ because in this case
the expression~\eqref{eq:LemmaX2KExpansionStatement} is just a standard
expansion of Cullis' determinant for $k = 2$ by Corollary~\ref{cor:CullisAsSumDet}.

Assume that $k \ge 3$. Let $A$ be defined by
\begin{equation*}
A = E_{(n-k+3)\, 1} + \ldots + E_{n\, (k-2)}.
\end{equation*}

It follows from the Laplace expansion of $\det _{n\, k} (X | A)$ along
the last column $k-2$ times that
\begin{equation*}
\det _{n\, k} (X | A) = (-1)^{s} \det _{(n-k+2)\, 2} (Y),
\end{equation*}
where $s \in \mathbb N$ and
$Y =
\begin{psmallmatrix}
x_{1\, 1} & x_{1\, 2}
\\
\vdots & \vdots
\\
x_{(n-k+2)\, 1} & x_{(n-k+2)\, 2}
\end{psmallmatrix}.$
By Corollary~\ref{cor:CullisAsSumDet} we obtain that
\begin{multline*}
\det _{(n-k+2)\, 2} (Y) =
\begin{vmatrix}
x_{1\, 1} & x_{1\, 2}
\\
x_{2\, 1} & x_{2\, 2}
\end{vmatrix} + \sum _{l = 3}^{n-k+2} (-1)^{l-2}
\begin{vmatrix}
x_{1\, 1} - x_{2\, 1}& x_{1\, 2} - x_{2\, 2}
\\
x_{l\, 1} & x_{l\, 2}
\end{vmatrix}
\\
+ \sum _{2 < l < m \le n-k+2} (-1)^{l + m - 3}
\begin{vmatrix}
x_{l\, 1} & x_{l\, 2}
\\
x_{m\, 1} & x_{m\, 2}
\end{vmatrix}.
\end{multline*}
From this we conclude that if $B$ is obtained from $A$ by multiplying its
last column by $(-1)^{s}$, then the desired equality~\eqref{eq:LemmaX2KExpansionStatement}
holds.
\end{proof}

Thus, using the expressions obtained above, we prove in the lemma below
that $\det _{n\, k} (X | B) = 0$ for all
$B \in {\mathcal M}_{n\, k-2}({\mathbb F})$ implies $\rk (X) \le 1$. This
lemma comprises the main part of the proof of the converse of Lemma~\ref{lem:Rank1CullisDeg1}.

\begin{lemma}%
\label{lem:AllZeroRank1}%
Assume that $k \ge 4$, $n \ge k + 2$ and $n + k$ is even. Let
$X \in {\mathcal M}_{n\, 2}({\mathbb F})$. If
%
\begin{equation}
\label{lem:AllZeroRank1:eqst}
\det _{n\, k} (X | A) = 0\;\;\text{for all}\;\;A \in {\mathcal M}_{n\, (k-2)}({
\mathbb F}),
\end{equation}
then $\rk (X) \le 1$.
\end{lemma}

\begin{proof}%
The proof is divided into three parts as follows.
\begin{enumerate}[label=(\roman{enumi}), ref=(\roman{enumi})]
\item
\label{lem:AllZeroRank1:part1}%
We prove that
$\rk
\begin{psmallmatrix}
x_{1\, 1} & x_{1\, 2}
\\
x_{2\, 1} & x_{2\, 2}%
\end{psmallmatrix}
\le 1$ for all $X$ satisfying the condition of the lemma.
\item
\label{lem:AllZeroRank1:part2}%
We prove that
$\rk
\begin{psmallmatrix}
x_{i\, 1} & x_{i\, 2}
\\
x_{(i+1)\, 1} & x_{(i+1)\, 2}%
\end{psmallmatrix}
\le 1$ for all $X$ satisfying the condition of the lemma and all
$1 \le i < n$.
\item
\label{lem:AllZeroRank1:part3}%
We prove that $\rk (X) \le 1$ for all $X$ satisfying the condition of the
lemma.
\end{enumerate}
\textit{\ref{lem:AllZeroRank1:part1}}\quad
By contradiction. Suppose that
$\rk
\begin{psmallmatrix}
x_{1\, 1} & x_{1\, 2}
\\
x_{2\, 1} & x_{2\, 2}%
\end{psmallmatrix}
= 2$. Let $\gamma _{1} = x_{1\, 1} - x_{2\,1}$ and
$\gamma _{2} = x_{1\, 2} - x_{2\, 2}$. Then
$(\gamma _{1},\gamma _{2}) \neq (0,0)$ because otherwise
$\rk
\begin{psmallmatrix}
x_{1\, 1} & x_{1\, 2}
\\
x_{2\, 1} & x_{2\, 2}%
\end{psmallmatrix}
\le 1$.

Let $2 < l < n$. By Lemma~\ref{lem:Exists2DiffDiff} there exists
$B \in {\mathcal M}_{n\, (k-2)}({\mathbb F})$ such that
\begin{equation*}
\det _{n\, k} (X | B) =
\begin{vmatrix}
x_{1\, 1} - x_{2\, 1} & x_{1\,2} - x_{2\, 2}
\\
x_{l\, 1} - x_{(l+1)\, 1} & x_{l\,2} - x_{(l+1)\, 2}
\end{vmatrix}
.
\end{equation*}
We obtain from our initial assumption that
$\det _{n\, k} (X | B) = 0$. Hence,
\begin{multline*}
\begin{vmatrix}
x_{1\, 1} - x_{2\, 1} & x_{1\,2} - x_{2\, 2}
\\
x_{l\, 1} - x_{(l+1)\, 1} & x_{l\,2} - x_{(l+1)\, 2}
\end{vmatrix}
=
\begin{vmatrix}
\gamma _{1} & \gamma _{2}
\\
x_{l\, 1} - x_{(l+1)\, 1} & x_{l\,2} - x_{(l+1)\, 2}
\end{vmatrix}
\\
= \gamma _{1} (x_{l\, 2}-x_{(l+1)\, 2}) - \gamma _{2} (x_{l\, 1} - x_{(l
+ 1)\, 1}) = 0,
\end{multline*}
and therefore
\begin{equation*}
\gamma _{1} x_{l\, 2} - \gamma _{2} x_{l\, 1} = \gamma _{1} x_{(l+1)
\, 2} - \gamma _{2} x_{(l + 1)\, 1}\;\;\text{for all}\;\; 2 < l < n.
\end{equation*}
Let us denote the common value of
$\gamma _{1} x_{l\, 2} - \gamma _{2} x_{l\, 1}$ by $v$.

By Lemma~\ref{lem:Exists2DiffSum} there exists
$B \in {\mathcal M}_{n\, (k-2)}({\mathbb F})$ such that
\begin{equation*}
\det _{n\, k} (X | B) = \sum _{2 < l \le n-k+3} (-1)^{l}
\begin{vmatrix}
x_{1\, 1} - x_{2\, 1} & x_{1\, 2} - x_{2\, 2}
\\
x_{l\, 1} & x_{l\, 2}
\end{vmatrix}
.
\end{equation*}

It follows from our initial assumption that
$\det _{n\, k} (X | B) = 0$. Hence,
\begin{multline*}
\sum _{2 < l \le n-k+3} (-1)^{l}
\begin{vmatrix}
x_{1\, 1} - x_{2\, 1} & x_{1\, 2} - x_{2\, 2}
\\
x_{l\, 1} & x_{l\, 2}
\end{vmatrix} = \sum _{2 < l \le n-k+3} (-1)^{l}
\begin{vmatrix}
\gamma _{1} & \gamma _{2}
\\
x_{l\, 1} & x_{l\, 2}
\end{vmatrix}
\\
= \sum _{2 < l \le n-k+3} (-1)^{l} (\gamma _{1} x_{l\, 2} - \gamma _{2}
x_{l\, 1}) = \sum _{2 < l \le n-k+3} (-1)^{l} v = 0.
\end{multline*}
Now, $\sum _{2 < l \le n-k+3} (-1)^{l} v = -v$ because the number of terms in the sum on the left-hand side of this equality
is odd and it begins with $-v$. Therefore, $v = 0$, i.e.
%
\begin{equation}
\label{eq:Gamma1Gamma2}
\gamma _{1} x_{i\,2} - \gamma _{2} x_{i\, 1} = 0\;\;\text{for all}\;\;3
\le i \le n.
\end{equation}
This provides us with a non-trivial linear relation between the columns of the
matrix
$
\begin{psmallmatrix}
x_{3\, 1} & x_{3\, 2}
\\
\vdots & \vdots
\\
x_{n\, 1} & x_{n\, 2}
\\
\end{psmallmatrix}
$. From this we conclude that
%
\begin{equation}
\label{eq:EveryDetZero}
\begin{vmatrix}
x_{l\, 1} & x_{l\, 2}
\\
x_{m\, 1} & x_{m\, 2}
\end{vmatrix}
= 0\;\;\text{for all}\;\;2 < l < m \le n.
\end{equation}

By Lemma~\ref{lem:Exists2Sum}, there exists
$B \in {\mathcal M}_{n,k-2}({\mathbb F})$ such that
\begin{multline*}
\det _{n\, k} (X | B) =
\begin{vmatrix}
x_{1\, 1} & x_{1\, 2}
\\
x_{2\, 1} & x_{2\, 2}
\end{vmatrix} + \sum _{l = 3}^{n-k+2} (-1)^{l-2}
\begin{vmatrix}
x_{1\, 1} - x_{2\, 1}& x_{1\, 2} - x_{2\, 2}
\\
x_{l\, 1} & x_{l\, 2}
\end{vmatrix}
\\
+ \sum _{2 < l < m \le n-k+2} (-1)^{l + m - 3}
\begin{vmatrix}
x_{l\, 1} & x_{l\, 2}
\\
x_{m\, 1} & x_{m\, 2}
\end{vmatrix}.
\end{multline*}
By~\eqref{eq:Gamma1Gamma2} and~\eqref{eq:EveryDetZero}, the last two summands
in this sum are equal to zero. By our initial assumption,
$\det _{n\, k} (X|B) = 0$. Hence,
\begin{equation*}
\det _{n\, k} (X | B) =
\begin{vmatrix}
x_{1\, 1} & x_{1\, 2}
\\
x_{2\, 1} & x_{2\, 2}
\end{vmatrix}
= 0,
\end{equation*}
which leads to a contradiction. Therefore,
$\rk
\begin{psmallmatrix}
x_{1\, 1} & x_{1\, 2}
\\
x_{2\, 1} & x_{2\, 2}%
\end{psmallmatrix}
\le 1$.

\bigskip
\noindent\textit{\ref{lem:AllZeroRank1:part2}}\quad
By contradiction. Suppose that
$X \in {\mathcal M}_{n\,2}({\mathbb F})$ satisfies the conditions of the
lemma and
$\rk
\begin{psmallmatrix}
x_{i\, 1} & x_{i\, 2}
\\
x_{(i+1)\, 1} & x_{(i+1)\, 2}%
\end{psmallmatrix}
= 2$ for some $1 \le i < n$.

Let $X' = {\mathbf S}_{i\,1}(X)$ where ${\mathbf S}_{i\,1}$ is defined
in Definition~\ref{def:SCSDef}. We claim that $X'$ also satisfies the conditions
of the lemma. Indeed, since ${\mathbf S}_{i\,1}$ commutes with the matrix
concatenation by Lemma~\ref{lem:SCShiftCommutesWithJoin} and
${\mathbf S}_{i\,1}$ is invertible by Lemma~\ref{lem:ReduceTo11ByShift},
then
\begin{equation*}
{\mathbf S}_{i\,1}(X) | {\mathbf S}_{i\,1}\left ({\mathbf S}_{i\,1}^{-1}(A)
\right ) = {\mathbf S}_{i\,1}\left (X | {\mathbf S}_{i\,1}^{-1}(A)
\right )\;\;\text{for all}\;\;A \in {\mathcal M}_{n\, (k-2)}({
\mathbb F}).
\end{equation*}
This implies that
%
\begin{multline}
\label{lem:AllZeroRank1:eq1}
\det _{n\,k} (X'|A) = \det _{n\,k}\left ({\mathbf S}_{i\,1}(X) | {
\mathbf S}_{i\,1}\left ({\mathbf S}_{i\,1}^{-1}(A)\right )\right )
\\
= \det _{n\,k}\left ({\mathbf S}_{i\,1}\left (X | {\mathbf S}_{i\,1}^{-1}(A)
\right )\right ) = \det _{n\,k}(X | {\mathbf S}^{-1}_{i\,1}(A))\;\;
\text{for all}\;\;A \in {\mathcal M}_{n\, (k-2)}({\mathbb F})
\end{multline}
because ${\mathbf S}_{i\,1}$ preserves $\det _{n\,k}$ by Lemma~\ref{lem:ReduceTo11ByShift}.
Next,
%
\begin{equation}
\label{lem:AllZeroRank1:eq3}
\det _{n\,k}\left (X | {\mathbf S}_{i\,1}^{-1}(A)\right ) = 0\;\;
\text{for all}\;\;A \in {\mathcal M}_{n\, (k-2)}({\mathbb F})
\end{equation}
because $X$ satisfies the conditions of the lemma. By substituting~\eqref{lem:AllZeroRank1:eq3} into~\eqref{lem:AllZeroRank1:eq1} we obtain that
%
\begin{equation}
\label{lem:AllZeroRank1:eq2}
\det _{n\,k} (X'|A) = 0\;\;\text{for all}\;\; A \in {\mathcal M}_{n\, (k-2)}({
\mathbb F}).
\end{equation}

The equality~\eqref{lem:AllZeroRank1:eq2} means that $X'$ satisfies the
conditions of the lemma. Hence,
$\rk
\begin{psmallmatrix}
x'_{1\, 1} & x'_{1\, 2}
\\
x'_{2\, 1} & x'_{2\, 2}
\end{psmallmatrix}
\le 1$ by the part~\ref{lem:AllZeroRank1:part1} of the lemma. Since
\begin{equation*}
\begin{pmatrix}
x_{i\, 1} & x_{i\, 2}
\\
x_{(i+1)\, 1} & x_{(i+1)\, 2}
\end{pmatrix}
= (-1)^{n-i}
\begin{pmatrix}
x'_{1\, 1} & x'_{1\, 2}
\\
x'_{2\, 1} & x'_{2\, 2}
\end{pmatrix}
\end{equation*}
by the definition of $X'$, then
$\rk
\begin{psmallmatrix}
x_{i\, 1} & x_{i\, 2}
\\
x_{(i+1)\, 1} & x_{(i+1)\, 2}
\end{psmallmatrix}
\le 1$. This leads to a contradiction. Therefore,
$\rk
\begin{psmallmatrix}
x_{i\, 1} & x_{i\, 2}
\\
x_{(i+1)\, 1} & x_{(i+1)\, 2}%
\end{psmallmatrix}
\le 1$ for all $1 \le i < n$.

\bigskip
\noindent\textit{\ref{lem:AllZeroRank1:part3}}\quad
This case is carried out similarly to the previous case. Suppose that
$X \in {\mathcal M}_{n\,2}({\mathbb F})$ satisfies the conditions of the
lemma and $\rk (X) = 2$. Hence,
$\rk
\begin{psmallmatrix}
x_{i\, 1} & x_{i\, 2}
\\
x_{j\, 1} & x_{j\, 2}%
\end{psmallmatrix}
= 2$ for some $1 \le i < j \le n$. Since the case $j = i + 1$ is carried
out in the part~\ref{lem:AllZeroRank1:part2} of the proof, we make the
stronger assumption that $j > i + 1$, or, equivalently,
$j - 1 \ge i + 1$.

Let
${\mathbf R}\colon {\mathcal M}_{n\,k}({\mathbb F}) \to {\mathcal M}_{n
\,k}({\mathbb F})$ be a linear map defined by
\begin{equation*}
{\mathbf R}= {\mathbf R}_{i+1} \circ {\mathbf R}_{i+2} \circ \cdots
\circ {\mathbf R}_{j-1},
\end{equation*}
where ${\mathbf R}_{i+1}, \ldots , {\mathbf R}_{j-1}$ are defined in Definition~\ref{def:RSwitchDef}
and $\circ $ denotes the composition operation. Let
$X' = {\mathbf R}(X)$. We claim that $X'$ also satisfies the conditions
of the lemma. Indeed, since
${\mathbf R}_{i+1}, \ldots , {\mathbf R}_{j-1}$ commute with  matrix
concatenation by Lemma~\ref{lem:RSwitchCommutesWithJoin} and are invertible
by Lemma~\ref{lem:RSwitchPreservesDetNK}, then
\begin{equation*}
{\mathbf R}(X) | {\mathbf R}\left ({\mathbf R}^{-1}(A)\right ) = {
\mathbf R}\left (X | {\mathbf R}^{-1}(A)\right )\;\;\text{for all}\;\;A
\in {\mathcal M}_{n\, (k-2)}({\mathbb F}).
\end{equation*}
This implies that
%
\begin{multline}
\label{lem:AllZeroRank1:eq1001}
\det _{n\,k} (X'|A) = \det _{n\,k} \left ({\mathbf R}(X) | {\mathbf R}
\left ({\mathbf R}^{-1}(A)\right )\right )
\\
= \det _{n\,k}\left ({\mathbf R}\left (X | {\mathbf R}^{-1}(A)\right )
\right ) = \det _{n\,k}(X |{\mathbf R}^{-1}(A))\;\;\mbox{for all}\;\;A \in {\mathcal M}_{n\, (k-2)}({\mathbb F}) 
\end{multline}
because
${\mathbf R}_{i+1}, \ldots , {\mathbf R}_{j-1}$ preserve
$\det _{n\,k}$ by Lemma~\ref{lem:RSwitchPreservesDetNK}. Next,
%
\begin{equation}
\label{lem:AllZeroRank1:eq1002}
\det _{n\,k}\left (X | {\mathbf R}(A)\right ) = 0\;\;\text{for all}\;
\;A \in {\mathcal M}_{n\, (k-2)}({\mathbb F})
\end{equation}
because $X$ satisfies the conditions of the lemma. By substituting~\eqref{lem:AllZeroRank1:eq1002}
 into~\eqref{lem:AllZeroRank1:eq1001} we obtain that
%
\begin{equation}
\label{lem:AllZeroRank1:eq1003}
\det _{n\,k} (X'|A) = 0\;\;\text{for all}\;\; A \in {\mathcal M}_{n\, (k-2)}({
\mathbb F}).
\end{equation}

The equality~\eqref{lem:AllZeroRank1:eq1003} means that $X'$ satisfies
the conditions of the lemma. Hence,
$\rk
\begin{psmallmatrix}
x'_{i\, 1} & x'_{i\, 2}
\\
x'_{(i+1)\, 1} & x'_{(i+1)\, 2}
\end{psmallmatrix}
\le 1$ by the part~\ref{lem:AllZeroRank1:part2} of the lemma. Since
\begin{equation*}
\begin{pmatrix}
x_{i\, 1} & x_{i\, 2}
\\
x_{j\, 1} & x_{j\, 2}
\end{pmatrix}
=
\begin{pmatrix}
x'_{i\, 1} & x'_{i\, 2}
\\
x'_{(i+1)\, 1} & x'_{(i+1)\, 2}
\end{pmatrix}
\end{equation*}
by the definition of $X'$, then
$\rk
\begin{psmallmatrix}
x_{i\, 1} & x_{i\, 2}
\\
x_{j\, 1} & x_{j\, 2}
\end{psmallmatrix}
\le 1$. This leads to a contradiction. Therefore, $\rk (X) \le 1$.
\end{proof}

\begin{remark}%
\label{rem5.8}
One can see that the statement of the part \ref{lem:AllZeroRank1:part1} of Lemma~\ref{lem:AllZeroRank1} can be reformulated
as follows
%
\begin{equation}
f_{1} = 0,\ldots , f_{N} = 0 \Rightarrow g = 0
\label{eq5.19}
\end{equation}
where $g$ denotes a polynomial representing the first $2\times 2$ minor
of $X$ and $f_{1}, \ldots , f_{N}$ denote the expressions used in the proof,
that is, $\det _{n\,k}(X|B)$ in the statements of Lemma~\ref{lem:Exists2DiffDiff},
Lemma~\ref{lem:Exists2DiffSum} and Lemma~\ref{lem:Exists2Sum}.

Indeed, it is possible to translate the proof of the part \ref{lem:AllZeroRank1:part1} of Lemma~\ref{lem:AllZeroRank1} into the polynomial language and show that $g^2 \in I$, where $I$ is an ideal in
${\mathbb F}[x_{1\,1}, x_{1\,2}, \ldots , x_{n\,1}, x_{n\,2}]$ generated
by $f_{1},\ldots , f_{N}$. This will imply that $g = 0$ by Lemma~\ref{lem2.12}.
%
%
%
%
\end{remark}

The last thing we need to show before we prove the converse of Lemma~\ref{lem:Rank1CullisDeg1}
is that every linear map preserving the Cullis' determinant does not increase
the value
$\max _{A}(\deg _{\lambda}(\det _{n\, k} (A + \lambda B)))$.

\begin{lemma}%
\label{lem:DetNKPresPresLowDeg}%
Assume that $k \ge 2,$ $n \ge k + 2$, $|{\mathbb F}| > k$ and $n + k$ is even. Let
$B \in {\mathcal M}_{n\,k}({\mathbb F})$ and
$T \colon {\mathcal M}_{n\,k}({\mathbb F}) \to {\mathcal M}_{n\,k}({
\mathbb F})$ be a linear map such that
$\det _{n\,k}(T(X)) = \det _{n\,k}(X)$ for all
$X \in {\mathcal M}_{n\,k}({\mathbb F})$. Then
%
\begin{equation}
\label{lem:DetNKPresPresLowDeg:eq1}
\deg _{\lambda}(\det _{n\, k} (A + \lambda B)) \le 1\quad\mbox{for all}\;\; A
\in {\mathcal M}_{n\,k}({\mathbb F})
\end{equation}
implies that
%
\begin{equation}
\label{lem:DetNKPresPresLowDeg:eq2}
\deg _{\lambda}(\det _{n\, k} (A + \lambda T(B))) \le 1\quad\mbox{for all}\;\; A
\in {\mathcal M}_{n\,k}({\mathbb F}).
\end{equation}
\end{lemma}

\begin{proof}%
By Corollary~\ref{cor:LinPresIso} every such $T$ is invertible. Since
$B$ satisfies the condition~\eqref{lem:DetNKPresPresLowDeg:eq1}, then in
particular
%
\begin{equation}
\label{lem:DetNKPresPresLowDeg:eq3}
\deg _{\lambda}(\det _{n\, k} (T^{-1}(A) + \lambda B)) \le 1\;\;\mbox{for all}\;\;A \in {\mathcal M}_{n\,k}({\mathbb F}).
\end{equation}

Now, for $T(B)$ we have
\begin{multline*}
\det _{n\, k} (A + \lambda T(B)) = \det _{n\, k} \left (T(T^{-1}(A)) +
\lambda T(B)\right )
\\
= \det _{n\, k} \left (T(T^{-1}(A) + \lambda B)\right ) = \det _{n\, k}
(T^{-1}(A) + \lambda B)\;\;\mbox{for all}\;\; A \in {\mathcal M}_{n\, k}({\mathbb F}),\,\lambda \in {\mathbb F}.
\end{multline*}
Thus, since
\begin{equation*}
\deg _{\lambda}\left (\det _{n\, k} (A + \lambda T(B))\right ) \le k
\;\;\text{and}\;\;\deg _{\lambda}\det _{n\, k} (T^{-1}(A) + \lambda B)
\le k
\end{equation*}
by Corollary~\ref{cor:DegDetABLEQK}, then
$\det _{n\, k} (A + \lambda T(B)) = \det _{n\, k} (T^{-1}(A) +
\lambda B)$ as elements of ${\mathbb F}[\lambda ]$ by Corollary~\ref{DegFGLessEqual}.
Hence,
\begin{equation*}
\deg _{\lambda}\left (\det _{n\, k} \left (A + \lambda T(B)\right )
\right ) = \deg _{\lambda}\left (\det _{n\, k} \left (T^{-1}(A) +
\lambda B\right )\right )\;\;\mbox{for all}\;\;A \in {\mathcal M}_{n\,k}({\mathbb F}). 
\end{equation*}
Therefore,
\begin{equation*}
\deg _{\lambda}\left (\det _{n\, k} \left (A + \lambda T(B)\right )
\right ) \le 1\;\;\text{for all}\;\;A \in {\mathcal M}_{n\,k}({
\mathbb F})
\end{equation*}
by~\eqref{lem:DetNKPresPresLowDeg:eq3}.
\end{proof}

Now we are ready to prove the converse of Lemma~\ref{lem:Rank1CullisDeg1}.

\begin{lemma}%
\label{lem:CullisDeg1Rank1}%
Assume that $|{\mathbb F}| > k \ge 4$, $n \ge k + 2$ and $n + k$ is even.
Let $B \in {\mathcal M}_{n\, k}({\mathbb F})$. If
\begin{equation*}
\deg _{\lambda} (\det _{n\, k} (A + \lambda B)) \le 1\;\;
\text{for all}\;\;A \in {\mathcal M}_{n\, k}({\mathbb F}),
\end{equation*}
then $\rk (B) \le 1$.
\end{lemma}

\begin{proof}
The proof is divided into the following two parts.
\begin{enumerate}[label=(\roman{enumi}), ref=(\roman{enumi})]
\item
\label{lem:CullisDeg1Rank1:part1}%
We prove that $\rk \bigl(B(|1,2]\bigr) \le 1$ for all
$B \in {\mathcal M}_{n\, k}({\mathbb F})$ satisfying the condition of the
lemma.
\item
\label{lem:CullisDeg1Rank1:part2}%
We prove that $\rk (B) \le 1$ for all
$B \in {\mathcal M}_{n\, k}({\mathbb F})$ satisfying the condition of the
lemma.
\end{enumerate}
\textit{\ref{lem:CullisDeg1Rank1:part1}}\quad
Let $B \in {\mathcal M}_{n\, k}({\mathbb F})$ be a matrix satisfying the
condition of the lemma. We show that
$B' = B(|1,2] = B(|1] \Big| B(|2]$ satisfies the condition of Lemma~\ref{lem:AllZeroRank1}.

For this, suppose that
$A \in {\mathcal M}_{n\, (k - 2)}({\mathbb F})$. Let
$A' \in {\mathcal M}_{n\, k}({\mathbb F})$ be defined by
\begin{equation*}
A' =
\begin{pmatrix}
O_{n\,1} & O_{n\,1} & A%
\end{pmatrix}
\end{equation*}
Consider the expansion~\eqref{eq:CullisBinomialExpansion} for
$P = \det _{n\, k} (A' + \lambda B) \in {\mathbb F}[\lambda ]$.
\begin{equation*}
P = \sum _{d = 0}^{k}\lambda ^{d} \sum _{1 \le i_{1} < \ldots < i_{d}
\le k} \det _{n\,k}\Bigl(A'(|1]\Big|\ldots \Big| B(|i_{1}] \Big|
\ldots \Big| B(|i_{d}] \Big| \ldots \Big| A'(|k] \Bigr)
\end{equation*}
Since $\deg _{\lambda}(P) \le 1$ for all
$A \in {\mathcal M}_{n\,k}({\mathbb F})$ by the condition of the lemma,
then the coefficient of $\lambda ^{d}$ in $P$ is zero for all
$d \ge 2$.

In particular, the coefficient of $\lambda ^{2}$ in $P$ is zero. This implies
that
%
\begin{equation}
\label{lem:CullisDeg1Rank1:eq1}
\sum _{1 \le i_{1} < i_{2} \le k} \det _{n\, k}\Bigl(A'(|1]\Big|
\ldots \Big| B(|i_{1}] \Big| \ldots \Big| B(|i_{2}] \Big| \ldots
\Big| A'(|k] \Bigr) = 0.
\end{equation}

Consider the particular summands of the above sum. Let
$1 \le i_{1} < i_{2} \le k$ and
$C \in {\mathcal M}_{n\,k}({\mathbb F})$ be defined by
\begin{equation*}
C = A'(|1]\Big|\ldots \Big| B(|i_{1}] \Big| \ldots \Big| B(|i_{2}]
\Big| \ldots \Big| A'(|k].
\end{equation*}
If $(i_{1},i_{2}) \neq (1,2)$, then the first or the second column of
$C$ contains $A'(|1] = 0$ or $A'(|2] = 0$ and consequently
$\det _{n\,k}(C) = 0$. Therefore,
\begin{multline*}
\sum _{1 \le i_{1} < i_{2} \le k} \det _{n\, k}\Bigl(A'(|1]\Big|
\ldots \Big| B(|i_{1}] \Big| \ldots \Big| B(|i_{2}] \Big| \ldots
\Big| A'(|k] \Bigr)
\\
= \det _{n\, k}\Bigl(B(|1] \Big| B(|2] \Big| A'(|3] \Big| \ldots
\Big| A'(|k] \Bigr)
\end{multline*}
By substituting this into~\eqref{lem:CullisDeg1Rank1:eq1} we obtain that
\begin{equation*}
\det _{n\, k}\Bigl(B(|1] \Big| B(|2] \Big| A'(|3] \Big| \ldots \Big| A'(|k]
\Bigr) = 0.
\end{equation*}
Since $A'(|3] \Big| \ldots \Big| A'(|k] = A$ and
$B' = B(|1] \Big| B(|2]$, then
$\det _{n\, k}(B'|A ) = 0.$
Hence, $B'$ satisfies the condition of Lemma~\ref{lem:AllZeroRank1} which implies that
$\rk \Bigl(B(|1,2]\Bigr) = \rk (B') \le 1.$

\textit{\ref{lem:CullisDeg1Rank1:part2}}\quad
By contradiction. Suppose that
$\deg _{\lambda} (\det _{n\, k} (A + \lambda B)) \le 1$ for every
$A \in {\mathcal M}_{n\, k}({\mathbb F})$ and $\rk (B) \ge 2$. Hence, there
exist $1 \le i < j \le k$ such that the $i$-th and the $j$-th column of
$B$ are linearly independent and consequently
%
\begin{equation}
\label{lem:CullisDeg1Rank1:eq3}
\rk \Bigl(B(|i,j]\Bigr) = 2.
\end{equation}

Since $i \neq j$, there exists a permutation $\sigma '$ sending $i$ to
1 and $j$ to 2. If $\sigma '$ is an even permutation, then let
$\sigma = \sigma '$. If $\sigma '$ is an odd permutation, then let
$\tau $ be a transposition of any two elements belonging to
$[k]\setminus \{i,j\}$. Such a transposition exists because
$k \ge 4$. Let $\sigma = \sigma ' \tau $. Thus, in any case
$\sigma $ is an even permutation which sends $i$ to 1 and $j$ to 2.

Let $C_{\sigma }\in {\mathcal M}_{k\,k}({\mathbb F})$ be a permutation
matrix corresponding to $\sigma $, $T_{\sigma}$ be an invertible linear
map on ${\mathcal M}_{n\,k}({\mathbb F})$ defined by
$T_{\sigma}(X) = XC_{\sigma}$ and $B' = BC_{\sigma }= T_{\sigma}(B)$.

Observe that,
%
\begin{equation}
\label{lem:CullisDeg1Rank1:eq2}
B'(|1, 2] = B(|i, j]
\end{equation}
by the definition of $\sigma $. In addition, $T_{\sigma}$ preserves
$\det _{n\,k}$ by Corollary~\ref{cor:rightmatrixmult} because
$\sigma $ is an even permutation and consequently
$\det _{k}(C_{\sigma}) = 1$.

By applying Lemma~\ref{lem:DetNKPresPresLowDeg} to $T_{\sigma}$ and
$B$ we conclude that
\begin{equation*}
\deg _{\lambda}(\det _{n\,k}(A + \lambda T_{\sigma}(B))) \le 1
\end{equation*}
for all $A \in {\mathcal M}_{n\,k}({\mathbb F})$. Hence, $B'$ satisfies
the condition of the lemma. The established part~\ref{lem:CullisDeg1Rank1:part1}
of the lemma implies that $\rk \Bigl(B'(|1,2]\Bigr) \le 1$. Therefore,
$\rk \Bigl(B(|i, j]\Bigr) \le 1$
by~\eqref{lem:CullisDeg1Rank1:eq2}, which contradicts~\eqref{lem:CullisDeg1Rank1:eq3}.
\end{proof}

The statement of Lemma~\ref{lem:CullisDeg1Rank1} does not hold if we do
not assume that $k \ge 4$ as it follows from the example below.

\begin{example}%
\label{cex:K3NKEven}%
Let $n \ge 5$ and 
$B =
\begin{psmallmatrix}
1 & 0 & 0 & 0 & -1 & 0 & \cdots & 0
\\
0 & 1 & 0 & -1 & 0 & 0 & \cdots & 0
\\
0 & 0 & 0 & 0 & 0 & 0 & \cdots & 0
\end{psmallmatrix}
^{t} \in {\mathcal M}_{n\,3}({\mathbb F})$. Then $\rk (B) = 2$ and
$\deg _{\lambda} (\det _{n\,3}(A + \lambda B)) \le 1$ for all
$A \in {\mathcal M}_{n\,3}({\mathbb F})$.
\end{example}

Now let us recall the following theorem due to Westwick reformulated in
order to apply it the particular case of maps on matrix spaces instead
of maps on tensor product spaces.

\begin{lemma}[\cite{westwick1967}, cf. Theorem 3.5]%
\label{lem:westwick}%
Assume that $n > k$ and let $T$ be an onto linear map of
${\mathcal M}_{n\,k}({\mathbb F})$ such that $\rk (T(X)) = 1$ for all
$X \in {\mathcal M}_{n\,k}({\mathbb F})$ with $\rk (X) = 1$. Then there
exist invertible matrices $A \in {\mathcal M}_{n\,n}({\mathbb F})$ and
$B \in {\mathcal M}_{k\,k}({\mathbb F})$ such that
\begin{equation*}
T(X) = AXB\;\;\text{for all}\;\;X \in {\mathcal M}_{n\,k}({\mathbb F}).
\end{equation*}
\end{lemma}

The lemma below shows that every linear map preserving the Cullis' determinant
satisfies the conditions of the Westwick's theorem.

\begin{lemma}%
\label{lem:LinPresRank}%
Assume that $|{\mathbb F}| > k \ge 4$, $n \ge k + 2$ and $n + k$ is even.
Let
$T\colon {\mathcal M}_{n\, k} ({\mathbb F}) \to {\mathcal M}_{n\, k} ({
\mathbb F})$ be a linear map such that
$\det _{n\, k} (T(X)) = \det _{n\,k}(X)$ for all
$X \in {\mathcal M}_{n\, k} ({\mathbb F})$. Then $\rk (X) = 1$ implies
$\rk (T(X)) = 1$.
\end{lemma}

\begin{proof}%
By Corollary~\ref{cor:LinPresIso} every such $T$ is invertible.

Suppose that $\rk (X) = 1$ (this implies that $X \neq 0$). Then
\begin{equation*}
\deg _{\lambda} (\det _{n\, k} (A + \lambda X)) \le 1\;\;
\text{for all}\;\; A \in {\mathcal M}_{n\, k}({\mathbb F})
\end{equation*}
by Lemma~\ref{lem:Rank1CullisDeg1}. Therefore,
\begin{equation*}
\deg _{\lambda} (\det _{n\, k} (A + \lambda T(X))) \le 1\;\;
\text{for all}\;\; A \in {\mathcal M}_{n\, k}({\mathbb F})
\end{equation*}
by Lemma~\ref{lem:DetNKPresPresLowDeg}. From this we conclude that
$\rk (T(X)) \le 1$ by Lemma~\ref{lem:CullisDeg1Rank1}. Since
$X \neq 0$ and $T$ is invertible, then $T(X) \neq 0$. Therefore,
$\rk (T(X)) = 1$.
\end{proof}

\begin{theorem}%
\label{thm:MainTheoremEvenKGe4}%
Assume that $|{\mathbb F}| > k \ge 4, n \ge k + 2$ and $n + k$ is even.
Let
$T$ be a linear map on ${\mathcal M}_{n\, k} ({\mathbb F})$. Then
$\det _{n\, k} (T(X)) = \det _{n\,k}(X)$ for all
$X \in {\mathcal M}_{n\, k} ({\mathbb F})$ if and only if there exist
$A \in {\mathcal M}_{n\, n}({\mathbb F})$ and
$B \in {\mathcal M}_{k\, k}({\mathbb F})$ such that $T(X) = AXB$ for all $X \in {\mathcal M}_{n\, k} ({\mathbb F})$ and
%
\begin{equation}
\label{thm:MainTheoremEvenKGe4:eq}
\det _{n\, k} \Bigl(A(|i_{1},\ldots , i_{k}]\Bigr) \cdot \det _{k}
\Bigl(B\Bigr) = (-1)^{\sum_{\alpha=1}^{k} (i_{\alpha} -\alpha)}\;\;\mbox{for all}\;\; 1 \le i_{1} < \ldots < i_{k} \le n.
\end{equation}
\end{theorem}
\begin{proof}%
The sufficiency follows from Lemma~\ref{lem:TwoSidedMulPreservesDet}. Let
us prove the necessity.

It follows from Lemma~\ref{lem:LinPresRank} that $\rk (T(X)) = 1$ for all
$X$ with $\rk (X) = 1$. Therefore, by Lemma~\ref{lem:westwick} there exist
invertible matrices $A \in {\mathcal M}_{n\,n}({\mathbb F})$ and
$B \in {\mathcal M}_{k\,k}({\mathbb F})$ such that
\begin{equation*}
T(X) = AXB\;\;\mbox{for all}\;\;X \in {\mathcal M}_{n\,k}({\mathbb F}). 
\end{equation*}
These $A$ and $B$ satisfy
the condition~\eqref{thm:MainTheoremEvenKGe4:eq} by Lemma~\ref{lem:TwoSidedMulPreservesDet}
because $T$ preserves $\det _{n\,k}$.
\end{proof}

\section{Linear maps preserving $\det _{n\,1}$}
\label{sec:K1}

Note that if $k = 1$, then the Cullis' determinant is just a linear function
and consequently the description of linear maps preserving $\det _{n\,1}$ is trivial as it is
shown in the following theorem.

\begin{theorem}%
\label{thm:MainTheoremK1}%
Let
$T\colon {\mathcal M}_{n\, 1} ({\mathbb F}) \to {\mathcal M}_{n\, 1} ({
\mathbb F})$ be a linear map. Then
\begin{equation*}
\det _{n\, 1} (T(X)) = \det _{n\,1}(X)\;\;\text{for all}\;\;X \in {
\mathcal M}_{n\, 1} ({\mathbb F})
\end{equation*}
if and only if there exists
$A \in {\mathcal M}_{n\, n}({\mathbb F})$ such that
%
\begin{equation}
\label{thm:MainTheoremK1:eq}
\det _{n\, 1} \Bigl(A(|i]\Bigr) = (-1)^{i - 1}\;\;\text{for all}\;\;1
\le i \le n
\end{equation}
and $T(X) = AX$ for all $X \in {\mathcal M}_{n\, 1} ({\mathbb F})$.
\end{theorem}
\begin{proof}%
Indeed, every linear map
$T\colon {\mathcal M}_{n\, 1} ({\mathbb F}) \to {\mathcal M}_{n\, 1} ({
\mathbb F})$ could be considered as a matrix multiplication for appropriate
$A$ because ${\mathcal M}_{n\,1}({\mathbb F}) = {\mathbb F}^{n}$. Thus,
there exists $A \in {\mathcal M}_{n\,n}({\mathbb F})$ such that
$T(X) = AX$ for all $X \in {\mathcal M}_{n\, 1} ({\mathbb F})$. Since this
map preserves $\det _{n\,1}$, then $A$ satisfies the condition~\eqref{thm:MainTheoremK1:eq}
by Lemma~\ref{lem:TwoSidedMulPreservesDet}.

Lemma~\ref{lem:TwoSidedMulPreservesDet} also implies that every map
$X \mapsto AX$ with $A$ satisfying the condition~\eqref{thm:MainTheoremK1:eq}
preserves $\det _{n\,1}$.
\end{proof}

\begin{remark}%
\label{rem6.2}
Since every linear map preserving $\det _{n\,1}$ is essentially the linear
map preserving the covector
$(1, -1, \ldots , (-1)^{(n-1)}) \in {{\mathbb F}^{n}}^{*}$, then there
exist noninvertible $\det _{n\,1}$-preservers for all $n > 1$.
\end{remark}

\section{Linear maps preserving $\det _{n\,2}$}
\label{sec:K2}

In this section we provide an example showing that for even $n$ a linear
map preserving $\det _{n\,2}$ may not correspond to the description of
linear maps preserving $\det _{n\,k}$ for $k \ge 3$ obtained in this paper
in its sequel. Note that this example is obtained by decomposing the quadratic
space
$\left ({\mathcal M}_{n\,2}({\mathbb F}), \det _{n\,2}\right )$ into an
appropriate direct sum of \emph{hyperbolic planes} and interchanging two
elements of the corresponding basis.

\begin{lemma}%
\label{lem:DetN2PresNonMatrix}%
Assume that $n$ is an even integer and $n \ge 2$. If
$$X =
\begin{pmatrix}
x_{1\,1} & x_{1\,2}
\\
\vdots & \vdots
\\
x_{n\,1} & x_{n\,2}%
\end{pmatrix}
\in {\mathcal M}_{n\,2}({\mathbb F}),\quad  S_{1} = \sum ^{n-1}_{i=2} (-1)^{i}
x_{i\,1},\quad S_{2} = \sum ^{n-1}_{i=2} (-1)^{i} x_{i\,2}$$
and $Y =
\begin{pmatrix}
S_{1} + S_{2} + x_{n\,2} & x_{1\,2}
\\
x_{2\,1} & x_{2\,2}
\\
\vdots & \vdots
\\
x_{(n-1)\,1} & x_{(n-1)\,2}
\\
x_{n\,1} & -S_{1} - S_{2} + x_{1\,1}
\end{pmatrix} \in {\mathcal M}_{n\,2}({\mathbb F}),$ then
$\det _{n\,2}(X) = \det _{n\,2}(Y).$
\end{lemma}
\begin{proof}
Let us express $\det _{n\,2}(X)$ and $\det _{n\,2}(Y)$ as sums of
$2\times 2$-determinants.
%
\begin{multline}
\label{lem:DetN2PresNonMatrix:eq1}
\det _{n\,2}(X) = \sum _{1 \le i < j \le n} (-1)^{i + j - 1 -2}
\begin{vmatrix}
x_{i\,1} & x_{i\,2}
\\
x_{j\,1} & x_{j\,2}
\end{vmatrix}
= \sum _{2 \le i < j \le n-1} (-1)^{i + j - 1 - 2}
\begin{vmatrix}
x_{i\,1} & x_{i\,2}
\\
x_{j\,1} & x_{j\,2}
\end{vmatrix}\\
+ \sum _{2 \le i \le n-1} (-1)^{i-2}
\begin{vmatrix}
x_{1\,1} & x_{1\,2}
\\
x_{i\,1} & x_{i\,2}
\end{vmatrix}
+ \left (\sum _{2 \le i \le n-1} (-1)^{i-1}
\begin{vmatrix}
x_{i\,1} & x_{i\,2}
\\
x_{n\,1} & x_{n\,2}
\end{vmatrix}\right ) +
\begin{vmatrix}
x_{1\,1} & x_{1\,2}
\\
x_{n\,1} & x_{n\,2}
\end{vmatrix}
\end{multline}
and
%
\begin{multline}
\label{lem:DetN2PresNonMatrix:eq2}
\det _{n\,2}(Y) = \sum _{1 \le i < j \le n} (-1)^{i + j - 1 - 2}
\begin{vmatrix}
Y_{i\,1} & Y_{i\,2}
\\
Y_{j\,1} & Y_{j\,2}
\end{vmatrix}
\\
=\sum _{2 \le i < j \le n-1} (-1)^{i + j - 1 - 2}
\begin{vmatrix}
x_{i\,1} & x_{i\,2}
\\
x_{j\,1} & x_{j\,2}
\end{vmatrix}
+ \sum _{2 \le i \le n-1} (-1)^{i-2}
\begin{vmatrix}
S_{1} + S_{2} + x_{n\,2} & x_{1\,2}
\\
x_{i\,1} & x_{i\,2}
\end{vmatrix}
\\
\phantom{XX}
+ \hspace{-5pt}\sum _{2 \le i \le n-1} (-1)^{i-1}
\begin{vmatrix}
x_{i\,1} & x_{i\,2}
\\
x_{n\,1} & -S_{1} - S_{2} + x_{1\,1}
\end{vmatrix}
+
\begin{vmatrix}
S_{1} + S_{2} + x_{n\,2} & x_{1\,2}
\\
x_{n\,1} & -S_{1} - S_{2} + x_{1\,1}
\end{vmatrix}.
\end{multline}
Since the first summands in~\eqref{lem:DetN2PresNonMatrix:eq1} and~\eqref{lem:DetN2PresNonMatrix:eq2}
are equal, then we need only to show that
%
\begin{multline}
\label{lem:DetN2PresNonMatrix:eq3}
\sum _{2 \le i \le n-1} (-1)^{i-2}
\begin{vmatrix}
x_{1\,1} & x_{1\,2}
\\
x_{i\,1} & x_{i\,2}
\end{vmatrix} + \left(\sum _{2 \le i \le n-1} (-1)^{i-1}
\begin{vmatrix}
x_{i\,1} & x_{i\,2}
\\
x_{n\,1} & x_{n\,2}
\end{vmatrix}\right) +
\begin{vmatrix}
x_{1\,1} & x_{1\,2}
\\
x_{n\,1} & x_{n\,2}
\end{vmatrix}
\\
= \sum _{2 \le i \le n-1} (-1)^{i-2}
\begin{vmatrix}
S_{1} + S_{2} + x_{n\,2} & x_{1\,2}
\\
x_{i\,1} & x_{i\,2}
\end{vmatrix}
+ \sum _{2 \le i \le n-1} (-1)^{i-1}
\begin{vmatrix}
x_{i\,1} & x_{i\,2}
\\
x_{n\,1} & -S_{1} - S_{2} + x_{1\,1}
\end{vmatrix}
\\
+
\begin{vmatrix}
S_{1} + S_{2} + x_{n\,2} & x_{1\,2}
\\
x_{n\,1} & -S_{1} - S_{2} + x_{1\,1}
\end{vmatrix}
\end{multline}

Let us simplify the first summand of the right-hand side of~\eqref{lem:DetN2PresNonMatrix:eq3}
using the multilinear and antisymmetric properties of $2\times 2$-determinant.
First,
\begin{equation*}
\sum _{i = 2}^{n-1} (-1)^{i-2}
\begin{vmatrix}
S_{1} + S_{2} + x_{n\,2} & x_{1\,2}
\\
x_{i\,1} & x_{i\,2}%
\end{vmatrix}
=
\begin{vmatrix}
S_{1} + S_{2} + x_{n\,2} & x_{1\,2}
\\
S_{1} & S_{2}%
\end{vmatrix}
\end{equation*}
by the multilinearity of $2\times 2$-determinant along the last row. Second,
\begin{equation*}
\begin{vmatrix}
S_{1} + S_{2} + x_{n\,2} & x_{1\,2}
\\
S_{1} & S_{2}%
\end{vmatrix}
=
\begin{vmatrix}
S_{2} + x_{n\,2} & x_{1\,2} - S_{2}
\\
S_{1} & S_{2}%
\end{vmatrix}
\end{equation*}
by the antisymmetricity of $2\times 2$-determinant along the rows. Third,
%
\begin{equation}
\label{lem:DetN2PresNonMatrix:eqq1}
\begin{vmatrix}
S_{2} + x_{n\,2} & x_{1\,2} - S_{2}
\\
S_{1} & S_{2}%
\end{vmatrix}
=
\begin{vmatrix}
S_{2} + x_{n\,2} & x_{1\,2} + x_{n\,2}
\\
S_{1} & S_{2} + S_{1}%
\end{vmatrix}
%
\end{equation}
by the antisymmetricity of $2\times 2$-determinant along the columns.

The second summand of the right-hand side of~\eqref{lem:DetN2PresNonMatrix:eq3}
could be simplified in the similar way as it is done below.
%
\begin{multline}
\label{lem:DetN2PresNonMatrix:eqq2}
\sum _{i = 2}^{n-1} (-1)^{i-1}
\begin{vmatrix}
x_{i\,1} & x_{i\,2}
\\
x_{n\,1} & -S_{1} -S_{2} + x_{1\,1}
\end{vmatrix} = -
\begin{vmatrix}
S_{1} & S_{2}
\\
x_{n\,1} & -S_{1} -S_{2} + x_{1\,1}
\end{vmatrix}
\\
= -
\begin{vmatrix}
S_{1} & S_{2}
\\
x_{n\,1} + S_{1}& -S_{1} + x_{1\,1}
\end{vmatrix} = -
\begin{vmatrix}
S_{1} & S_{1} + S_{2}
\\
x_{n\,1} + S_{1} & x_{n\,1} + x_{1\,1}
\end{vmatrix}.
\end{multline}

Now let us find the sum of the right-hand sides of~\eqref{lem:DetN2PresNonMatrix:eqq1}
and~\eqref{lem:DetN2PresNonMatrix:eqq2} (simplified versions of the first
and the second summand of the right-hand side of~\eqref{lem:DetN2PresNonMatrix:eq3})
using the multilinear and antisymmetric properties of $2\times 2$-determinant
%
\begin{multline}
\label{lem:DetN2PresNonMatrix:eqq3}
\begin{vmatrix}
S_{2} + x_{n\,2} & x_{1\,2} + x_{n\,2}
\\
S_{1} & S_{2} + S_{1}
\end{vmatrix} -
\begin{vmatrix}
S_{1} & S_{1} + S_{2}
\\
x_{n\,1} + S_{1} & x_{n\,1} + x_{1\,1}
\end{vmatrix}
\\
=
\begin{vmatrix}
S_{2} + x_{n\,2} & x_{1\,2} + x_{n\,2}
\\
S_{1} & S_{2} + S_{1}
\end{vmatrix} +
\begin{vmatrix}
x_{n\,1} + S_{1} & x_{n\,1} + x_{1\,1}
\\
S_{1} & S_{1} + S_{2}
\end{vmatrix}
\\
=
\begin{vmatrix}
S_{1} + S_{2} + x_{n\,1} + x_{n\,2} & x_{1\,1} + x_{1\,2} + x_{n\,1} +
x_{n\,2}
\\
S_{1} & S_{2} + S_{1}
\end{vmatrix}
\\
=
\begin{vmatrix}
S_{1} + S_{2} + x_{n\,1} + x_{n\,2} & x_{1\,1} + x_{1\,2} - S_{1} - S_{2}
\\
S_{1} & S_{2}
\end{vmatrix}.
\end{multline}

Let us add to the right-hand side of~\eqref{lem:DetN2PresNonMatrix:eqq3}
the third summand of the right-hand side of~\eqref{lem:DetN2PresNonMatrix:eq3}
and simplify it. 
\begin{multline*}
\begin{vmatrix}
S_{1} + S_{2} + x_{n\,1} + x_{n\,2} & x_{1\,1} + x_{1\,2} - S_{1} - S_{2}
\\
S_{1} & S_{2}
\end{vmatrix}
+
\begin{vmatrix}
S_{1} + S_{2} + x_{n\,2} & x_{1\,2}
\\
x_{n\,1} & -S_{1} -S_{2} + x_{1\,1}
\end{vmatrix}
\\
= (S_{1} + S_{2} + x_{n\,1} + x_{n\,2})S_{2} - (x_{1\,1} + x_{1\,2} - S_{1}
- S_{2})S_{1}
\phantom{XXXXXXXXX}
\\
+ (S_{1} + S_{2} + x_{n\,2})(-S_{1} -S_{2} + x_{1\,1}) - x_{1\,2}x_{n
\,1}
\\
= (S_{1} + S_{2})^{2} + (x_{n\,1} + x_{n\,2})S_{2} - (x_{1\,1} + x_{1
\,2})S_{1} - (S_{1}+S_{2})^{2}
\\
+ (S_{1}+S_{2})x_{1\,1} - (S_{1}+S_{2})x_{n\,2} + x_{n\,2}x_{1\,1} - x_{1
\,2} x_{n\,1}.
\end{multline*}
The resulting expression could be transformed as follows
\begin{multline*}
(S_{1} + S_{2})^{2} + (x_{n\,1} + x_{n\,2})S_{2} - (x_{1\,1} + x_{1\,2})S_{1}
- (S_{1}+S_{2})^{2}
\\
+ (S_{1}+S_{2})x_{1\,1} - (S_{1}+S_{2})x_{n\,2} + x_{n\,2}x_{1\,1} - x_{1
\,2} x_{n\,1}
\\
= (x_{n\,1} + x_{n\,2})S_{2} - (x_{1\,1} + x_{1\,2})S_{1} + (S_{1}+S_{2})x_{1
\,1} - (S_{1}+S_{2})x_{n\,2} +
\begin{vmatrix}
x_{1\,1} & x_{1\,2}
\\
x_{n\,1}& x_{n\,2}
\end{vmatrix}
\\
= x_{1\,1} S_{2} - x_{1\,2} S_{1} + x_{n\,1} S_{2} - x_{n\,2}S_{1} +
\begin{vmatrix}
x_{1\,1} & x_{1\,2}
\\
x_{n\,1} & x_{n\,2}
\end{vmatrix}
\phantom{XXXXXXXXXXX}
\\
=
\begin{vmatrix}
x_{1\,1} & x_{1\,2}
\\
S_{1} & S_{2}
\end{vmatrix} -
\begin{vmatrix}
S_{1} & S_{2}
\\
x_{n\,1} & x_{n\,2}
\end{vmatrix} +
\begin{vmatrix}
x_{1\,1} & x_{1\,2}
\\
x_{n\,1} & x_{n\,2}
\end{vmatrix}
\phantom{XXXXXXXXXXXX}
\\
\phantom{XX}
= \sum \limits _{i = 2}^{n-1} (-1)^{i - 2}
\begin{vmatrix}
x_{1\,1} & x_{1\,2}
\\
x_{i\,1} & x_{i\,2}
\end{vmatrix} + \sum \limits _{i = 2}^{n-1} (-1)^{i - 1}
\begin{vmatrix}
x_{i\,1} & x_{i\,2}
\\
x_{n\,1} & x_{n\,2}
\end{vmatrix} +
\begin{vmatrix}
x_{1\,1} & x_{1\,2}
\\
x_{n\,1} & x_{n\,2}
\end{vmatrix}.
\end{multline*}%
 This is equal to the left-hand side of~\eqref{lem:DetN2PresNonMatrix:eq3}.
Therefore, the equality~\eqref{lem:DetN2PresNonMatrix:eq3} holds and consequently
$\det _{n\,2}(X) = \det _{n\,2}(Y).$
\end{proof}

\begin{lemma}%
\label{lem:CounterEXK2NEven}%
Assume that $n$ is an even integer and $n \ge 4$. Then a linear map
$\mathbf{T}' \colon {\mathcal M}_{n\,2}({\mathbb F}) \to {\mathcal M}_{n
\,2}({\mathbb F})$ defined by
\begin{equation*}
\mathbf{T}'(X) =
\begin{pmatrix}
\sum \limits ^{n-1}_{i=2} (-1)^{i} x_{i\,1} + \sum \limits ^{n}_{i=2} (-1)^{i}
x_{i\,2} & x_{1\,2}
\\
x_{2\,1} & x_{2\,2}
\\
\vdots & \vdots
\\
x_{(n-1)\,1} & x_{(n-1)\,2}
\\
x_{n\,1} & \sum \limits ^{n-1}_{i=2} (-1)^{i-1} x_{i\,2} + \sum
\limits ^{n-1}_{i=1} (-1)^{i-1} x_{i\,1}
\end{pmatrix}
\end{equation*}
is a linear map preserving $\det _{n\,2}$ which does not have the form
$X \mapsto AXB$ for any $A \in {\mathcal M}_{n\,n}({\mathbb F})$,
$B \in {\mathcal M}_{2\,2}({\mathbb F})$.
\end{lemma}
\begin{proof}%
Let us show that $\mathbf{T}'$ preserves $\det _{n\,2}$. Indeed, for a
given matrix $X \in {\mathcal M}_{n\,2}({\mathbb F})$ the matrix
$\mathbf{T}'(X)$ is equal to $Y$ defined in Lemma~\ref{lem:DetN2PresNonMatrix}
and consequently $\det _{n\,2}(\mathbf{T}'(X)) = \det _{n\,2}(X)$.

Let us prove by contradiction that $\mathbf{T}'$ does not have the form
$X \mapsto AXB$ for any $A \in {\mathcal M}_{n\,n}({\mathbb F})$,
$B \in {\mathcal M}_{2\,2}({\mathbb F})$. Suppose that
$\mathbf{T}'(X) = AXB$ for some
$A \in {\mathcal M}_{n\,n}({\mathbb F})$,
$B \in {\mathcal M}_{2\,2}({\mathbb F})$ and all
$X \in {\mathcal M}_{n\,2}({\mathbb F})$. Then the properties of rank of
the product of matrices imply that
%
\begin{equation}
\label{lem:CounterEXK2NEven:eq1}
\rk (\mathbf{T}'(X)) = \rk (AXB) \le \rk (X) \quad \text{for all}\;\;X
\in {\mathcal M}_{n\,2}({\mathbb F}).
\end{equation}

Let $X' = E_{2\,1}$. Then
\begin{equation*}
\mathbf{T}'(X') = E_{1\,1} + E_{2\,1} - E_{n\,2}
\end{equation*}
and clearly $\rk \left (\mathbf{T}'(X')\right ) = 2$. This contradicts~\eqref{lem:CounterEXK2NEven:eq1}.
Thus, $\mathbf{T}'$ does not have the desired form.
\end{proof}

\begin{remark}
\label{rem7.3}
Note that if $k = 2$, $n \ge k$ and $n$ is even, then $\det _{n\,2}$ is
a non-degenerate quadratic form. This implies in particular that every
linear map preserving $\det _{n\,2}$ could be expressed as a composition
of reflections, except for the case when $n = 2$ and ${\mathbb F} = {\mathbb Z}_2$ (see~\cite[\S 4. Witt's Theorem]{Milnor2013}). One may see
that the quadratic space
$\left ({\mathcal M}_{n\,2}({\mathbb F}), \det _{n\,2}\right )$ is in fact
a \emph{split} (or \emph{metabolic}) quadratic space, i.e. has a totally
isotropic subspace of dimension $n$ spanned by matrices
$E_{1\,1}, \ldots , E_{n\,1}$ (see~\cite[\S 6. Split Inner Product Spaces]{Milnor2013}).
\end{remark}
\section*{Declaration of competing interest}

The authors declare that they have no known competing financial interests or personal relationships that could have appeared to influence the work reported in this paper.

\section*{Acknowledgements}

We thank the anonymous referee for the careful reading, many helpful comments and suggestions. The research of the second author was supported by the scholarship of the Center for Absorption in Science, the Ministry for Absorption of Aliyah, the State of Israel.

\section*{Data availability}

No data was used for the research described in the article.

\bibliographystyle{plainnat}

\end{document}